\documentclass[11pt,a4paper]{amsart}

\usepackage[T1]{fontenc}
\usepackage[utf8]{inputenc}
\usepackage{amsmath,amssymb,amsthm,mathtools}
\usepackage{geometry}
\usepackage{hyperref}
\usepackage{setspace}
\newtheorem{theorem}{Theorem}[section]
\newtheorem{proposition}[theorem]{Proposition}
\newtheorem{lemma}[theorem]{Lemma}
\newtheorem{corollary}[theorem]{Corollary}
\newtheorem{definition}[theorem]{Definition}
\newtheorem{remark}[theorem]{Remark}
\newtheorem{assumption}[theorem]{Assumption}

\newcommand{\T}{\mathbb T}
\newcommand{\R}{\mathbb R}
\newcommand{\C}{\mathbb C}
\newcommand{\N}{\mathbb N}
\newcommand{\Z}{\mathbb Z}
\newcommand{\PP}{\mathbb P}
\newcommand{\Leb}{\operatorname{Leb}}

\newcommand{\Gr}{\operatorname{Gr}}
\newcommand{\cE}{\mathcal E}

\newcommand{\norm}[1]{\left\lVert #1\right\rVert}
\newcommand{\gap}{\Delta_{\mathrm{sp}}}
\newcommand{\Prob}{\operatorname{Prob}}

\title{Analyticity of Lyapunov Exponents for Mixed
Markov--Quasi-Periodic Cocycles}

\author[E.H. Y. Tall]{El Hadji Yaya Tall}
\address{Centro de investigación de Matemática Pura y Aplicada, Universidad de Costa Rica. San José, Costa Rica.}
\email{el.tallL@ucr.ac.cr}
\date{\today}

\begin{document}
\begin{abstract}
We study the dependence of the top Lyapunov exponent of a mixed
Markov quasi-periodic cocycle on the transition probabilities.
The transition matrix is assumed primitive, the Markov extension on
the torus is assumed non-resonant along directed cycles, and the top
Lyapunov exponent is assumed simple.  We first prove the result in the
irreducible case by a forward Markov transfer operator and contraction
estimates uniform over the initial Markov state.  The reducible case is
then obtained, as in the Bernoulli argument of
Bezerra-Sánchez-Tall, by decomposing along measurable invariant
sections into restricted and quotient bundle cocycles and inducting on
the fiber dimension.  A measurable invariant section need not admit a
continuous trivialization, even when the original matrices cocycle are
continuous.  We therefore formulate the irreducible analytic case for
essentially bounded measurable bundle cocycles with essentially bounded
inverses
\end{abstract}
\maketitle
\tableofcontents
\section{Introduction}

One of the central questions in the theory of linear cocycles is the
\emph{regularity problem}: how do the Lyapunov exponents depend on the
cocycle and on the driving dynamical system?  Besides measuring the
asymptotic growth of matrix products, Lyapunov exponents play a basic
role in smooth ergodic theory, random dynamical systems, and the
spectral theory of random and quasi-periodic operators.  Their
dependence on parameters has therefore been studied since
Furstenberg's foundational work on noncommuting random
products~\cite{Furstenberg63}.

For random matrix products, Furstenberg and Kifer~\cite{FK} and
Hennion~\cite{Hennion84} established early continuity results under
irreducibility-type hypotheses.  Bocker and Viana~\cite{BV17}
proved continuity for compactly supported random
\(\operatorname{GL}(2)\)-cocycles, and Malheiro and Viana~\cite{MV}
developed the corresponding stationary measure vector method for
finite-state Markov shifts.  Continuity results were subsequently
extended to cocycles with invariant holonomies by
Backes, Brown and Butler~\cite{BBB18}, while
Avila, Eskin and Viana~\cite{AEV23} obtained continuity for compactly
supported random products in arbitrary dimension.  The compactness
assumption is significant: Sánchez and Viana~\cite{SV20} show that in the
non-compactly supported setting continuity may fail.

Quantitative regularity is more delicate.  Le~Page~\cite{LePage}
proved Hölder continuity of the top exponent under strong
irreducibility and contraction assumptions.  In dimension two,
Tall--Viana~\cite{TV20} obtained pointwise Hölder continuity at every
compactly supported distribution with simple spectrum and a
log-Hölder modulus without the simplicity assumption. For finitely
supported measures with distinct exponents, Duarte and Klein~\cite{DK20} obtained, weak-Hölder continuity. More recently, Araújo~\cite{Araujo26} extended the pointwise Hölder conclusion of
Tall--Viana to random products in arbitrary fiber dimension, assuming
that the top exponent is simple and that the associated equator has dimension at most one.  Here the equator is the maximal invariant subspace on which the asymptotic growth is strictly smaller than the top exponent.

A parallel theory concerns quasi-periodic cocycles over irrational
translations of the torus, which arise naturally in the study of
quasi-periodic Schrödinger operators.  Regularity in this deterministic
setting is notably rigid and delicate; see, among others,
Avila, Jitomirskaya and Sadel~\cite{AJS14} and
Duarte and Klein~\cite{DK14,DK19}.  Mixed random quasi-periodic cocycles
were introduced by Cai, Duarte and Klein~\cite{CDKMixed} to combine
random selection with quasi-periodic motion in the base.  Their later
Furstenberg theory~\cite{CDK} provides uniform directional growth and
projective contraction estimates for the Bernoulli version of this
model.

Analytic dependence has a distinct history.  Ruelle~\cite{Ruelle}
used transfer operator methods to prove analyticity under suitable
hyperbolicity assumptions.  Peres~\cite{Peres} proved analytic
dependence on transition probabilities for finite random matrix
products, including Markov dependence under the appropriate
irreducibility and simplicity hypotheses.  More recently,
Amorim, Durães and Melo~\cite{ADM26} extended this circle of ideas to
compact, possibly infinite, symbol spaces using perturbation theory
for Markov operators on Banach spaces.  In the mixed
random quasi-periodic setting, Bezerra, Sánchez and Tall~\cite{BST}
proved real-analytic dependence on the Bernoulli probabilities when
the top Lyapunov exponent is simple.

The purpose of this work is to combine the Markov and mixed
random quasi-periodic directions.  Let \(X=\{1,\dots,N\}\), let
\(P=(p_{ij})\) be a finite-state transition matrix, and choose
translations \(\theta_i\in\T^m\) and continuous maps
\[
 A_i:\T^m\longrightarrow\operatorname{GL}_d(\R).
\]
On the one-sided Markov shift, consider
\[
 F_{\mathbf A}(x,t,v)
 =\bigl(\sigma x,t+\theta_{x_0},A_{x_0}(t)v\bigr).
\]
We keep \(\mathbf A=(A_i)_i\), the translations
\(\boldsymbol\theta=(\theta_i)_i\), and the support of \(P\) fixed,
and study the dependence of the top Lyapunov exponent on the positive
transition probabilities.

Two new features distinguish this problem from the Bernoulli model.
First, the transfer operator and every stationary object must retain
the current Markov state.  A single stationary projective measure is
replaced by a finite vector of conditional measures, and estimates
must remain uniform after conditioning on the initial state.  Second,
ergodicity of the torus extension is no longer detected by an
individual translation: the exact obstruction is a Fourier resonance
along an admissible directed cycle.  We therefore impose a cycle
non-resonance condition and prove that it is equivalent to ergodicity
of the stationary Markov torus skew product.

Our main theorem is local in the relative interior of a fixed support
face.  Under primitivity, cycle non-resonance, and simplicity of the
top exponent, it gives a holomorphic extension of the top Lyapunov
exponent to a complex neighborhood of the reference transition
matrix.  No irreducibility assumption is imposed in the final
statement.  We first prove an irreducible core using conditional
projective contraction and a direct Weierstrass argument.  Reducible
cocycles are then decomposed along state indexed measurable invariant
families into restricted and quotient bundle cocycles.  Simplicity
selects a unique locally dominant block, and induction on the fiber
dimension completes the proof.  Because measurable invariant families
need not possess continuous trivializations, the terminal irreducible
case is formulated for essentially bounded measurable bundle
cocycles with essentially bounded inverses.

Keeping the support fixed is essential for this reduction: the same
invariant family then remains invariant for every nearby stochastic
matrix. The proof is also deliberately direct. 
Finite transfer iterates are polynomial in the complexified transition entries,
projective contraction makes the logarithmic increments locally
uniformly Cauchy, and the Weierstrass theorem produces the holomorphic
limit.

\subsection*{Organization of the paper}

Section~2 describes the Markov base.  Its main purpose is
to identify the exact obstruction to ergodicity of the torus
extension. 
Section~3 introduces the cocycle and states the main result.  In Section~4 we replace the stationary contraction average used in the Bernoulli case by an estimate that is uniform over the initial Markov state.  This estimate is also taken uniformly over the torus point; both uniformities are essential for the submultiplicative and Cauchy
arguments.  The proof is organized as a finite state Markov adaptation
of the uniform averaged-convergence mechanism developed by
Cai, Duarte and Klein for Bernoulli mixed random quasi-periodic
cocycles~\cite{CDK}.
Sections~5 and~6 construct the complex transfer operator and a
uniformly convergent sequence of holomorphic logarithmic increments.
  
Section~7 carries out
the reducible induction in the Markov bundle setting.  It also records
the measurable bounded version of the irreducible case needed after
restriction and passage to a quotient; this is a separate argument,
not a formal consequence of proof for continuous cocycles.

\subsection*{Standing conventions}

All vector spaces are finite dimensional and endowed with Euclidean
norms.  Haar probability measure on \(\T^m\) is denoted by \(\Leb\).
Statements involving projective points are independent of the choice
of non-zero representatives.  A neighborhood inside \(H(U)\) always

\section{Statement and main results}
\subsection{The stationary Markov measure}

Let \(\Sigma=X^\N\) and let \(\sigma:\Sigma\to\Sigma\) be the
left shift.  A stochastic matrix \(P\) is called \emph{primitive} if
there exists \(n_0\ge1\) for which every entry of \(P^{n_0}\) is
strictly positive.

For a primitive \(P\), let \(\pi(P)\) be the unique stationary row
vector:
\[
 \pi(P)P=\pi(P),\qquad
 \pi_i(P)>0,\qquad
 \sum_i\pi_i(P)=1.
\]
The stationary Markov measure \(\mu_P\) is determined by
\[
 \mu_P([i_0,\dots,i_n])
 =\pi_{i_0}(P)p_{i_0i_1}\cdots p_{i_{n-1}i_n}.
\]

For each \(i\in X\), set
\[
 \Sigma_i=\{x\in\Sigma:x_0=i\}
\]
and denote by \(\mu_{P,i}\) the Markov path measure conditioned on
the initial state \(i\):
\[
 \mu_{P,i}:=\mu_P(\,\cdot\mid x_0=i).
\]
Thus, if \(i_0=i\),
\begin{equation}\label{eq:conditional-path-measure}
 \mu_{P,i}([i_0,\ldots,i_n])
 =
 p_{i_0i_1}\cdots p_{i_{n-1}i_n}.
\end{equation}
All integrals over \(\Sigma_i\) below are taken with respect to this
probability measure.  For a different transition matrix \(Q\), the
analogous conditional path measure is denoted by \(\mu_{Q,i}\).

\begin{lemma}\label{lem:stationary-analytic}
The map \(P\mapsto\pi(P)\) is real-analytic on every relative open
set of irreducible stochastic matrices with fixed support.
\end{lemma}

\begin{proof}
The vector \(\pi(P)^T\) is the unique solution of
\[
 (P^T-I)\pi^T=0,\qquad \mathbf1^T\pi^T=1.
\]
Replace one redundant row of \(P^T-I\) by \(\mathbf1^T\).  The
resulting square matrix is invertible because the eigenvalue \(1\) is
simple.  Cramer's rule expresses each coordinate of \(\pi(P)\) as a
rational function of the free entries of \(P\), with non-vanishing
denominator near the given matrix.
\end{proof}

\subsection{Cycle non-resonance}

Define
\[
 F(x,t)=(\sigma x,t+\theta_{x_0}).
\]
The measure \(\mu_P\times\Leb\) is \(F\)-invariant.  Ergodicity of this
extension is not implied merely by the irrationality of one
translation.  The correct condition must take the transition graph
into account.

For \(k\in\Z^m\), introduce the diagonal unitary matrix
\[
 D_k=\operatorname{diag}
 \left(
 e^{2\pi i\langle k,\theta_1\rangle},\ldots,
 e^{2\pi i\langle k,\theta_N\rangle}
 \right)
\]
and the twisted transition matrix
\[
 P_k=D_kP.
\]

\begin{theorem}\label{thm:ergodicity-criterion}
Suppose \(P\) is irreducible.  The following conditions are equivalent.
\begin{enumerate}
 \item The skew product
 \[
 F(x,t)=(\sigma x,t+\theta_{x_0})
 \]
 is ergodic with respect to \(\mu_P\times\Leb\).
 \item For every \(0\ne k\in\Z^m\),
 \[
 1\notin\operatorname{spec}(D_kP).
 \]
 \item For every \(0\ne k\in\Z^m\), there is an admissible directed
 cycle \(c=(i_0,\ldots,i_{\ell-1},i_\ell=i_0)\) such that
 \[
 \left\langle k,\sum_{r=0}^{\ell-1}\theta_{i_r}\right\rangle
 \notin\Z.
 \]
\end{enumerate}
\end{theorem}

\begin{remark}[Relation with the Bernoulli criterion]
In the Bernoulli case the transition matrix has identical rows,
\(P=\mathbf e p^T\), where \(p=(p_1,\ldots,p_N)^T\).  Therefore
\(D_kP=(D_k\mathbf e)p^T\) has rank one, and its only possibly
non-zero eigenvalue is
\[
 p^TD_k\mathbf e
 =
 \sum_{i=1}^N p_i
 e^{2\pi i\langle k,\theta_i\rangle}.
\]
Since the coefficients \(p_i\) are positive, this eigenvalue is equal
to \(1\) if and only if
\[
 \langle k,\theta_i\rangle\in\mathbb Z
 \qquad\text{for every }i.
\]
Thus Theorem~\ref{thm:ergodicity-criterion} reduces to the
Bernoulli criterion of Cai--Duarte--Klein
\cite[Proposition~2.1 and Theorem~2.3]{CDKMixed}: for every
\(k\in\mathbb Z^m\setminus\{0\}\), there must exist a frequency
\(\theta_i\) in the support of the Bernoulli law such that
\[
 \langle k,\theta_i\rangle\notin\mathbb Z.
\]
The cycle formulation is the corresponding extension to a 
Markov support graph, in which one-symbol loops need not be
admissible.
\end{remark}

\begin{proof}

\smallskip
\noindent $(1) \iff (2)$:
let \(h\in L^2(\mu_P\times\Leb)\) be \(F\)-invariant.  Disintegrating
with respect to the present state and taking Fourier coefficients in
the torus variable gives, for each \(k\in\Z^m\), a vector
\[
 a(k)=(a_1(k),\ldots,a_N(k))^T\in\C^N.
\]
The invariance equation says
\[
 a_i(k)
 =e^{2\pi i\langle k,\theta_i\rangle}
   \sum_{j=1}^Np_{ij}a_j(k),
 \qquad i=1,\ldots,N,
\]
or equivalently
\[
 a(k)=D_kPa(k).
\]
For \(k=0\), irreducibility of the stationary Markov shift implies
that the fixed vectors are constant.  Therefore all \(F\)-invariant
functions are constant exactly when \(D_kP\) has no fixed vector for
every \(k\ne0\).  This proves the equivalence of (1) and (2).

\smallskip
\noindent $(3) \implies (2)$:
suppose \(a=D_kPa\) with \(a\ne0\).  Let
\[
 R=\max_i|a_i|
\]
and choose \(i\) with \(|a_i|=R\).  Then
\[
 R=|a_i|
 \le\sum_jp_{ij}|a_j|\le R.
\]
Every inequality is an equality.  Hence \(|a_j|=R\) for every
successor \(j\) of \(i\), and all complex numbers
\[
 e^{2\pi i\langle k,\theta_i\rangle}a_j
\]
with \(p_{ij}>0\) have the same argument.  Irreducibility propagates
this conclusion through the whole graph.  Thus \(R>0\), all
\(|a_i|=R\), and for every admissible edge \(i\to j\),
\[
 a_i=e^{2\pi i\langle k,\theta_i\rangle}a_j.
\]
Multiplying these identities around a cycle
\(i_0\to\cdots\to i_{\ell-1}\to i_0\) yields
\[
 1=
 e^{2\pi i\langle k,\theta_{i_0}+\cdots+\theta_{i_{\ell-1}}\rangle}.
\]
Thus every directed cycle is resonant.

\smallskip
\noindent $(2) \implies (3)$:
conversely, suppose all directed cycles are resonant for a fixed
\(k\ne0\).  Fix a reference vertex \(i_*\) and set \(a_{i_*}=1\).
For any vertex \(j\), choose an admissible path
\[
 i_*=i_0\to i_1\to\cdots\to i_s=j
\]
and define
\[
 a_j=
 \exp\left(
 -2\pi i\left\langle k,
 \sum_{r=0}^{s-1}\theta_{i_r}\right\rangle
 \right).
\]
This does not depend on the chosen path.  Indeed, irreducibility
allows each of two candidate paths to be closed by a path from \(j\)
back to \(i_*\); the quotient of the two definitions is the phase
accumulated around a pair of directed cycles, and both phases equal
one.  For every edge \(i\to j\) we consequently have
\[
 a_i=e^{2\pi i\langle k,\theta_i\rangle}a_j.
\]
Since all successors of \(i\) give the same \(a_j\),
\[
 e^{2\pi i\langle k,\theta_i\rangle}
 \sum_jp_{ij}a_j=a_i.
\]
Therefore \(a=D_kPa\), so \(1\in\operatorname{spec}(D_kP)\).
\end{proof}

\begin{assumption}[Cycle non-resonance]\label{ass:cycle}
For every \(0\ne k\in\Z^m\), there is a directed cycle
\[
 c=(i_0,i_1,\dots,i_{\ell-1},i_\ell=i_0)
\]
in the support graph of \(P\) such that
\[
 \left\langle k,\sum_{r=0}^{\ell-1}\theta_{i_r}\right\rangle
 \notin\Z.
\]
\end{assumption}

\begin{proposition}\label{prop:base-ergodic}
If \(P\) is irreducible and Assumption~\ref{ass:cycle} holds, then
\((F,\mu_P\times\Leb)\) is ergodic.
\end{proposition}

\begin{proof}
This is the implication (3)\(\Rightarrow\)(1) in
Theorem~\ref{thm:ergodicity-criterion}.
\end{proof}

\begin{corollary}\label{cor:self-loop}
Suppose \(P\) is irreducible and \(p_{jj}>0\) for some state \(j\).
If \(\theta_j\) is rationally independent, then
\((F,\mu_P\times\Leb)\) is ergodic.
\end{corollary}

\begin{proof}
The self-loop \(j\to j\) is a directed cycle whose displacement is
\(\theta_j\).  For \(k\ne0\), rational independence gives
\(\langle k,\theta_j\rangle\notin\Z\).  Apply
Theorem~\ref{thm:ergodicity-criterion}.
\end{proof}

\begin{remark}
If \(P\) is strictly positive, every state has a self-loop.  In that
common special case, the original and simpler assumption that at least
one translation is rationally independent is sufficient.  The full
cycle criterion is needed only because a primitive matrix may contain
zeros and the irrational state may fail to lie on a one-edge cycle.
\end{remark}

\section{Cocycles and the main result}
For \(x=(x_n)_{n\ge0}\), define
\[
 t_n=t+\sum_{r=0}^{n-1}\theta_{x_r},\qquad
 A_x^n(t)=A_{x_{n-1}}(t_{n-1})\cdots A_{x_0}(t).
\]
By Oseledets' theorem, the ergodic system of
Proposition~\ref{prop:base-ergodic} has Lyapunov exponents
\[
 \lambda_1(P)\ge\cdots\ge\lambda_d(P).
\]
We write \(\lambda_+(P)=\lambda_1(P)\) and
\(\gap=\lambda_1(P)-\lambda_2(P)\).

Because the symbol set and the torus are compact and each \(A_i\) is
continuous with values in \(\operatorname{GL}_d(\R)\), both
\[
 \log^+\norm{A_{x_0}(t)}
 \quad\text{and}\quad
 \log^+\norm{A_{x_0}(t)^{-1}}
\]
are bounded.  Thus all integrability assumptions in the invertible
version of Oseledets' theorem are automatic.  For almost every
\((x,t)\), there is a filtration
\[
 \{0\}=V_{\ell+1}(x,t)\subsetneq V_\ell(x,t)
 \subsetneq\cdots\subsetneq V_1(x,t)=\R^d
\]
such that every non-zero
\(v\in V_r(x,t)\setminus V_{r+1}(x,t)\) satisfies
\[
 \lim_{n\to\infty}\frac1n\log\norm{A_x^n(t)v}
 =\lambda_r(P),
\]
where repetitions are included according to multiplicity.
\begin{definition}\label{def:irreducible}
The Markov family
\((\mathbf A,\boldsymbol\theta,U)\) is \emph{irreducible} if there is
no collection of measurable maps
\[
 \mathcal V_i:\T^m\to\Gr(k,d),\qquad i\in X,\quad 0<k<d,
\]
satisfying
\[
 A_i(t)\mathcal V_i(t)=\mathcal V_j(t+\theta_i)
\]
for every admissible edge \(i\to j\) and for
Lebesgue-almost every \(t\).
\end{definition}

Kifer's non-random filtration for random bundle maps leads to the
weaker notion of \emph{quasi-irreducibility}: there is no proper
measurable invariant subbundle whose Lyapunov exponent is strictly
smaller than \(\lambda_1(P)\); see
\cite[Chapter~III]{Kifer}.  This terminology is also used in the
Bernoulli formulation of Cai--Duarte--Klein
\cite[Definition~3.4]{CDK}.  In the present Markov setting the
subbundles are state-indexed and their invariance equations are
imposed only along admissible edges.  Irreducibility implies
quasi-irreducibility.  We keep the stronger hypothesis in the main
irreducible theorem to avoid an additional layer of notation, but the
proof of the uniform-growth lemma below uses only the corresponding
Markov quasi-irreducibility property.

Let \(U=(u_{ij})\) be the support matrix of \(P\), and put
\[
 \cE=\{(i,j):u_{ij}=1\}.
\]
We shall also use a measurable category that is stable under
restriction to invariant families and passage to quotients.  For
\(r\geq1\), let \(\mathcal B_r(U)\) consist of collections

\[
 \mathbf B=(B_{ij})_{(i,j)\in\cE},
 \qquad
 B_{ij}:\mathbb T^m\longrightarrow\operatorname{GL}_r(\mathbb R),
\]
whose entries are measurable and satisfy
\begin{equation}\label{eq:measurable-bounded-class}
 \operatorname*{ess\,sup}_{(i,j)\in\cE,\,t\in\mathbb T^m}
 \max\bigl\{\|B_{ij}(t)\|,\|B_{ij}(t)^{-1}\|\bigr\}<\infty.
\end{equation}
The map attached to the edge \(i\to j\) sends the fiber over
\((i,t)\) to the fiber over \((j,t+\theta_i)\).  The original
continuous cocycle is a special element of this class, with
\(B_{ij}(t)=A_i(t)\).  The edge notation is necessary after a
state-indexed invariant family has been trivialized: the resulting
coordinate matrix may then depend on both endpoints of the edge.

Continuity and measurable boundedness will play different roles.
Continuity makes the projective kernel Feller and gives estimates
uniform in the initial torus point.  The analytic argument needed for
dimension reduction uses only torus-integrated estimates and will be
proved separately for cocycles in \(\mathcal B_r(U)\).  In particular,
the measurable result below is not obtained by approximating a
measurable cocycle by continuous ones.

Define
\[
 S(U)=\left\{Q=(q_{ij}):
 q_{ij}>0\text{ on }\cE,\ q_{ij}=0\text{ off }\cE,\
 \sum_jq_{ij}=1\right\},
\]
\[
 H(U)=\left\{Z=(z_{ij})\in\C^{N\times N}:
 z_{ij}=0\text{ off }\cE,\ \sum_jz_{ij}=1\right\}.
\]

\begin{theorem}\label{thm:irreducible-core}
Let \(P\in S(U)\) be primitive and suppose
Assumption~\ref{ass:cycle} holds.  Let
\(A_i\in C^0(\T^m,\operatorname{GL}_d(\R))\), and assume that
\((\mathbf A,\boldsymbol\theta,U)\) is irreducible and
\[
 \lambda_1(P)>\lambda_2(P).
\]
Then there are a relatively open neighborhood
\(\Omega\subset H(U)\) of \(P\) and a holomorphic function
\(\Lambda:\Omega\to\C\) such that
\[
 \Lambda(Q)=\lambda_+(Q)
\]
for every \(Q\in\Omega\cap S(U)\).
\end{theorem}

\begin{theorem}[Main theorem]\label{thm:main}
Let \(P\in S(U)\) be primitive and suppose
Assumption~\ref{ass:cycle} holds.  Let
\(A_i\in C^0(\T^m,\operatorname{GL}_d(\R))\), and assume only that
\[
 \lambda_1(P)>\lambda_2(P).
\]
Then there are a relatively open neighborhood
\(\Omega\subset H(U)\) of \(P\) and a holomorphic function
\(\Lambda:\Omega\to\C\) such that
\[
 \Lambda(Q)=\lambda_+(Q)
\]
for every \(Q\in\Omega\cap S(U)\).
\end{theorem}

\begin{remark}[Two irreducible inputs]\label{rem:two-cores}
The proof uses two related but distinct forms of the irreducible
argument.  For the original continuous cocycle,
Lemma~\ref{lem:uniform-growth} below gives convergence uniformly in
the state, torus point, and projective direction; its proof uses the
Markov projective kernel.  After dimension reduction the restriction
and quotient maps are generally only measurable.  For those cocycles
we use Proposition~\ref{prop:bundle-core}, whose estimates are
integrated over the torus and require only
\eqref{eq:measurable-bounded-class}.  No uniform-in-\(t\) conclusion
is asserted in that measurable category.
\end{remark}

\begin{corollary}\label{cor:all-exponents}
If the full spectrum at \(P\) is simple, then every Lyapunov exponent
is real-analytic near \(P\) in \(S(U)\).
\end{corollary}

\begin{proof}
For \(1\le k\le d\),
\[
 \lambda_1(P)+\cdots+\lambda_k(P)
 =\lambda_+\!\left(\bigwedge^k\mathbf A,P\right).
\]
Apply Theorem~\ref{thm:main} to consecutive exterior powers and use
\[
 \lambda_k(P)
 =\lambda_+\!\left(\bigwedge^k\mathbf A,P\right)
 -\lambda_+\!\left(\bigwedge^{k-1}\mathbf A,P\right).
\]
\end{proof}

\section{Conditional projective contraction}

On \(\PP^{d-1}\), use the distance
\[
 d(\bar u,\bar v)
 =\frac{\norm{u\wedge v}}{\norm u\,\norm v}.
\]
For \(\alpha>0\), define the contraction quantity uniformly in the
initial Markov state and the initial torus point:
\[
 \widehat K_n(\alpha,P)
 =\max_{i\in X}\sup_{t\in\T^m}\sup_{u\ne v}
 \int_{\Sigma_i}
 \left(
 \frac{d(A_x^n(t)u,A_x^n(t)v)}{d(u,v)}
 \right)^\alpha
 d\mu_{P,i}(x),
\]
where \(\mu_{P,i}\) is the conditional path measure defined in
\eqref{eq:conditional-path-measure}.


\begin{remark}
   A bound obtained only after averaging with respect to the stationary
Markov measure,
\[
 \int_\Sigma H\,d\mu_P
 =
 \sum_{j=1}^N\pi_j
 \int_{\Sigma_j}H\,d\mu_{P,j},
\]
does not provide a bound for each state-conditioned integral
\[
 \int_{\Sigma_j}H\,d\mu_{P,j}.
\]
After conditioning on the prefix
\((x_0,\ldots,x_\ell)\), the future starts from the particular state
\(x_\ell\), and hence its law is \(\mu_{P,x_\ell}\).  Therefore the
inner conditional integral must be estimated uniformly over all
possible values of \(x_\ell\).  This explains the maximum over the
initial state in the definition of \(\widehat K_n(\alpha,P)\).

The supremum over \(t\) is equally important.  After the first block,
the second block begins at the random point
\[
 t_\ell=t+\sum_{r=0}^{\ell-1}\theta_{x_r}.
\]
Moreover, in the Cauchy estimate of
Proposition~\ref{prop:direct-cauchy}, one of the two initial
projective directions is \(A_i(s-\theta_i)v\), which depends on the
integration variable \(s\).  A coefficient in which \(t\) is
integrated before the supremum over directions does not control this
situation.  The pointwise in \(t\) definition above does.
\end{remark}

\begin{lemma}[Submultiplicativity]\label{lem:submult}
For all \(n,\ell\ge0\),
\[
 \widehat K_{n+\ell}(\alpha,P)
 \le \widehat K_n(\alpha,P)\widehat K_\ell(\alpha,P).
\]
\end{lemma}

\begin{proof}
Fix \(n,\ell\geq1\) and \(\alpha>0\).  We prove that
\[
 \widehat K_{n+\ell}(\alpha,P)
 \leq
 \widehat K_n(\alpha,P)\,
 \widehat K_\ell(\alpha,P).
\]

Fix an initial state \(i\in X\), a torus point \(t\in\T^m\), and
distinct projective points \(u,v\in\PP^{d-1}\).  We want to  estimate
\[
 \int_{\Sigma_i}
 \left(
 \frac{
 d(A_x^{n+\ell}(t)u,A_x^{n+\ell}(t)v)}
 {d(u,v)}
 \right)^\alpha
 d\mu_{P,i}(x).
\]

The cocycle identity gives
\begin{equation}\label{eq:cocycle-splitting-projective}
 A_x^{n+\ell}(t)
 =
 A_{\sigma^\ell x}^{n}(t_\ell)\,
 A_x^\ell(t).
\end{equation}
Define the projective contraction ratio of the first \(\ell\) steps by
\[
 R_\ell(x;t,u,v)
 =
 \frac{
 d(A_x^\ell(t)u,A_x^\ell(t)v)}
 {d(u,v)}.
\]
Since every matrix \(A_i(t)\) is invertible and \(u\neq v\), the
projective points
\[
 u_\ell=A_x^\ell(t)u,
 \qquad
 v_\ell=A_x^\ell(t)v
\]
are still distinct.  We may therefore define the contraction ratio of
the remaining \(n\) steps by
\[
 R_n'(\sigma^\ell x;t_\ell,u_\ell,v_\ell)
 =
 \frac{
 d\bigl(
 A_{\sigma^\ell x}^{n}(t_\ell)u_\ell,
 A_{\sigma^\ell x}^{n}(t_\ell)v_\ell
 \bigr)}
 {d(u_\ell,v_\ell)}.
\]
Using \eqref{eq:cocycle-splitting-projective}, we obtain 
$$
 \frac{
 d(A_x^{n+\ell}(t)u,A_x^{n+\ell}(t)v)}
 {d(u,v)} =  R_n'(\sigma^\ell x;t_\ell,u_\ell,v_\ell) R_\ell(x;t,u,v).
$$
After raising both sides to the power \(\alpha\), this becomes
\begin{equation}\label{eq:factor-projective-ratio}
 \left(
 \frac{
 d(A_x^{n+\ell}(t)u,A_x^{n+\ell}(t)v)}
 {d(u,v)}
 \right)^\alpha
 =
 R_\ell(x;t,u,v)^\alpha
 R_n'(\sigma^\ell x;t_\ell,u_\ell,v_\ell)^\alpha.
\end{equation}
Under
\(\mu_{P,i}\), the initial coordinate is \(x_0=i\), and a prefix
\[
 (x_0,x_1,\ldots,x_\ell)
 =
 (i,i_1,\ldots,i_\ell)
\]
has probability
\[
 p_{ii_1}p_{i_1i_2}\cdots p_{i_{\ell-1}i_\ell}.
\]
Conditioned on this prefix, the path after time \(\ell\) is a Markov
path starting at \(i_\ell\).  In other words, its conditional
distribution is \(\mu_{P,i_\ell}\). More precisely, if \(H\) is a non-negative measurable function of the
prefix and of the future path, then
\begin{align}
 &\int_{\Sigma_i}
 H(x_0,\ldots,x_\ell;\sigma^\ell x)
 \,d\mu_{P,i}(x)
 \notag\\
 &\quad=
 \sum_{i_1,\ldots,i_\ell}
 p_{ii_1}p_{i_1i_2}\cdots p_{i_{\ell-1}i_\ell}
 \int_{\Sigma_{i_\ell}}
 H(i,i_1,\ldots,i_\ell;y)
 \,d\mu_{P,i_\ell}(y),
 \label{eq:explicit-markov-disintegration}
\end{align}
where the sum is taken over all admissible prefixes and \(y_0=i_\ell\).

Once a prefix \((i,i_1,\ldots,i_\ell)\) is fixed, the following
quantities are fixed:
 $R_\ell,\, x_\ell=i_\ell,\, t_\ell,\, u_\ell,\, v_\ell.$
Only the future path remains to be integrated.  By the definition of
\(\widehat K_n(\alpha,P)\), applied with initial state \(i_\ell\),
initial torus point \(t_\ell\), and initial projective points
\(u_\ell,v_\ell\), we have
\begin{align}
 &\int_{\Sigma_{i_\ell}}
 \left(
 \frac{
 d\bigl(A_y^n(t_\ell)u_\ell,
         A_y^n(t_\ell)v_\ell\bigr)}
 {d(u_\ell,v_\ell)}
 \right)^\alpha
 d\mu_{P,i_\ell}(y)
 \notag\\
 &\qquad\leq
 \widehat K_n(\alpha,P).
 \label{eq:conditional-Kn-bound}
\end{align}
The estimate is applicable because the definition of
\(\widehat K_n\) takes the supremum over every possible initial
state, torus point, and pair of distinct projective directions.

Applying the disintegration formula
\eqref{eq:explicit-markov-disintegration} to
\eqref{eq:factor-projective-ratio}, and then using
\eqref{eq:conditional-Kn-bound}, gives
\begin{align*}
 &\int_{\Sigma_i}
 \left(
 \frac{
 d(A_x^{n+\ell}(t)u,A_x^{n+\ell}(t)v)}
 {d(u,v)}
 \right)^\alpha
 d\mu_{P,i}(x)
 \\
 &=
 \sum_{i_1,\ldots,i_\ell}
 p_{ii_1}\cdots p_{i_{\ell-1}i_\ell}
 R_\ell(i,i_1,\ldots,i_{\ell-1};t,u,v)^\alpha
 \\
 &\qquad\qquad\times
 \int_{\Sigma_{i_\ell}}
 \left(
 \frac{
 d\bigl(A_y^n(t_\ell)u_\ell,
         A_y^n(t_\ell)v_\ell\bigr)}
 {d(u_\ell,v_\ell)}
 \right)^\alpha
 d\mu_{P,i_\ell}(y)
 \\
 &\leq
 \widehat K_n(\alpha,P)
 \sum_{i_1,\ldots,i_\ell}
 p_{ii_1}\cdots p_{i_{\ell-1}i_\ell}
 R_\ell(i,i_1,\ldots,i_{\ell-1};t,u,v)^\alpha.
\end{align*}
The last sum is exactly
\[
 \int_{\Sigma_i}
 R_\ell(x;t,u,v)^\alpha
 \,d\mu_{P,i}(x).
\]
Consequently,
\begin{align*}
 &\int_{\Sigma_i}
 \left(
 \frac{
 d(A_x^{n+\ell}(t)u,A_x^{n+\ell}(t)v)}
 {d(u,v)}
 \right)^\alpha
 d\mu_{P,i}(x)
 \\
 &\qquad\leq
 \widehat K_n(\alpha,P)
 \int_{\Sigma_i}
 \left(
 \frac{
 d(A_x^\ell(t)u,A_x^\ell(t)v)}
 {d(u,v)}
 \right)^\alpha
 d\mu_{P,i}(x)
 \\
 &\qquad\leq
 \widehat K_n(\alpha,P)
 \widehat K_\ell(\alpha,P).
\end{align*}
The second inequality follows directly from the definition of
\(\widehat K_\ell(\alpha,P)\).

The initial data \(i\), \(t\), \(u\), and \(v\) were arbitrary.
Taking the maximum over \(i\) and the suprema over \(t\) and
\(u\neq v\) yields
\[
 \widehat K_{n+\ell}(\alpha,P)
 \leq
 \widehat K_n(\alpha,P)
 \widehat K_\ell(\alpha,P),
\]
as claimed.
%
\end{proof}

\subsection{Uniform averaged growth}

\subsection{Stationary unit vectors}

We use the stationary measure-vector notation of
Malheiro--Viana~\cite[Section~3]{MV}.  A vector
\[
 \boldsymbol\eta=(\eta_i)_{i\in X}
\]
of positive measures on \(\T^m\) is called a \emph{unit vector} if
every component \(\eta_i\) is a probability measure.  Define the
operator \(\mathcal P_P\) on unit vectors by
\begin{equation}\label{eq:measure-vector-operator}
 (\mathcal P_P\boldsymbol\eta)_j
 =
 \sum_{i=1}^N
 \frac{\pi_i(P)p_{ij}}{\pi_j(P)}
 (f_i)_*\eta_i,
 \qquad j\in X,
\end{equation}
where \(f_i(t)=t+\theta_i\).  The coefficients form a probability
vector for each fixed \(j\), because
\[
 \sum_i\frac{\pi_i(P)p_{ij}}{\pi_j(P)}
 =
 \frac{(\boldsymbol\pi(P)P)_j}{\pi_j(P)}=1.
\]
We say that \(\boldsymbol\eta\) is a
\emph{\(\mathcal P_P\)-stationary unit vector} if
\(\mathcal P_P\boldsymbol\eta=\boldsymbol\eta\).

Associated with a unit vector is the probability measure
\begin{equation}\label{eq:skew-product-measure-vector}
 \mu_P\ltimes\boldsymbol\eta
 =
 \sum_{i=1}^N
 \bigl(\mu_P|_{[0;i]}\bigr)\times\eta_i
\end{equation}
on \(\Sigma\times\T^m\).  Here
\(\mu_P|_{[0;i]}\) has total mass \(\pi_i(P)\).  A direct calculation
on cylinders shows that
\(\mu_P\ltimes\boldsymbol\eta\) is invariant under
\[
 F(x,t)=(\sigma x,t+\theta_{x_0})
\]
if and only if \(\boldsymbol\eta\) is
\(\mathcal P_P\)-stationary.  Thus \(\eta_i\) is precisely the
conditional distribution of the torus coordinate given \(x_0=i\).

\begin{proposition}[Uniqueness of the stationary torus unit vector]
\label{prop:unique-torus-vector}
Assume that \(P\) is irreducible and that
Assumption~\ref{ass:cycle} holds.  Then the unique
\(\mathcal P_P\)-stationary unit vector is
\[
 \boldsymbol\eta^*
 =(\Leb,\ldots,\Leb).
\]
Equivalently, \(\mu_P\times\Leb\) is the unique \(F\)-invariant
probability whose marginal on \(\Sigma\) is \(\mu_P\) and whose
conditional torus measure depends only on the present state.
\end{proposition}

\begin{proof}
Because every translation preserves Haar measure,
\(\boldsymbol\eta^*=(\Leb,\ldots,\Leb)\) is stationary.
We prove uniqueness by Fourier analysis.

Let \(\boldsymbol\eta=(\eta_i)_i\) be any stationary unit vector and
define
\[
 \widehat\eta_i(k)
 =
 \int_{\T^m}e^{-2\pi i\langle k,t\rangle}\,d\eta_i(t),
 \qquad k\in\Z^m.
\]
Taking the \(k\)-th Fourier coefficient in
\eqref{eq:measure-vector-operator} gives
\begin{equation}\label{eq:reversed-fourier-equation}
 \widehat\eta_j(k)
 =
 \sum_i
 \frac{\pi_i(P)p_{ij}}{\pi_j(P)}
 e^{-2\pi i\langle k,\theta_i\rangle}
 \widehat\eta_i(k).
\end{equation}
Introduce the reversed transition matrix
\[
 C_{ji}
 =
 \frac{\pi_i(P)p_{ij}}{\pi_j(P)}
\]
and the diagonal unitary matrix
\[
 D_k^-=
 \operatorname{diag}\bigl(
 e^{-2\pi i\langle k,\theta_1\rangle},
 \ldots,
 e^{-2\pi i\langle k,\theta_N\rangle}\bigr).
\]
The matrix \(C\) is row-stochastic and irreducible, with reversed
edges \(j\to i\) precisely when \(p_{ij}>0\).  In vector notation,
\eqref{eq:reversed-fourier-equation} is
\[
 \widehat{\boldsymbol\eta}(k)
 =CD_k^-\widehat{\boldsymbol\eta}(k).
\]

Fix \(k\ne0\) and suppose that
\(\widehat{\boldsymbol\eta}(k)\ne0\).  Put
\[
 R=\max_i|\widehat\eta_i(k)|
\]
and choose \(j\) for which
\(|\widehat\eta_j(k)|=R\).  Equation
\eqref{eq:reversed-fourier-equation} gives
\[
\begin{aligned}
 R
 &=|\widehat\eta_j(k)|\\
 &\le
 \sum_iC_{ji}|\widehat\eta_i(k)|
 \le R.
\end{aligned}
\]
Every inequality is therefore an equality.  Hence every predecessor
\(i\) with \(p_{ij}>0\) has
\(|\widehat\eta_i(k)|=R\).  Irreducibility propagates this fact to all
vertices, so \(R>0\) and
\[
 |\widehat\eta_i(k)|=R
 \quad\text{for every }i.
\]
Equality in the triangle inequality also implies that, for every
admissible edge \(i\to j\),
\begin{equation}\label{eq:fourier-edge-phase}
 \widehat\eta_j(k)
 =
 e^{-2\pi i\langle k,\theta_i\rangle}
 \widehat\eta_i(k).
\end{equation}

Let
\[
 i_0\to i_1\to\cdots\to i_{\ell-1}\to i_\ell=i_0
\]
be any admissible directed cycle.  Multiplying
\eqref{eq:fourier-edge-phase} around the cycle yields
\[
 1=
 \exp\left(
 -2\pi i
 \left\langle k,
 \sum_{r=0}^{\ell-1}\theta_{i_r}
 \right\rangle\right).
\]
Thus every admissible directed cycle is resonant for \(k\), contrary
to Assumption~\ref{ass:cycle}.  Therefore
\[
 \widehat\eta_i(k)=0
 \quad\text{for every }i\text{ and every }k\ne0.
\]
Since each \(\eta_i\) is a probability measure,
\(\widehat\eta_i(0)=1\).  These are exactly the Fourier coefficients
of Haar measure.  Uniqueness of Fourier coefficients on \(\T^m\)
gives \(\eta_i=\Leb\) for every \(i\).
\end{proof}

We record separately the uniform-growth input used in the contraction
argument.  Cai, Duarte and Klein prove the Bernoulli version, uniformly
in the torus point and projective direction, in
\cite[Theorem~5.4]{CDK}.  The next proof records the finite-state
Markov adaptation.

\begin{lemma}[Uniform averaged growth]\label{lem:uniform-growth}
Assume the hypotheses of
Theorem~\ref{thm:irreducible-core}.  Then
\[
 L_n(i,t,v):=
 \frac1n\int_{\Sigma_i}
 \log\frac{\norm{A_x^n(t)v}}{\norm v}\,d\mu_{P,i}(x)
\]
converges to \(\lambda_1(P)\), uniformly in
\((i,t,v)\in X\times\T^m\times\PP^{d-1}\).  Moreover,
\[
 \limsup_{n\to\infty}\ 
 \max_i\sup_t\frac1n
 \int_{\Sigma_i}
 \log\norm{\textstyle\bigwedge^2 A_x^n(t)}
 \,d\mu_{P,i}(x)
 \le\lambda_1(P)+\lambda_2(P).
\]
\end{lemma}

\begin{proof}

\noindent\emph{Stationary measure vectors for the projective
cocycle.}
Let
\[
 \mathcal Z=\T^m\times\PP^{d-1}.
\]
For each \(i\in X\), define
\[
 F_i:\mathcal Z\longrightarrow\mathcal Z,
 \qquad
 F_i(t,[v])
 =
 \bigl(t+\theta_i,[A_i(t)v]\bigr).
\]
Let $\boldsymbol\eta=(\eta_i)_{i\in X}$ be a vector of finite positive measures on \(\mathcal Z\).  Following
the measure-vector notation of Malheiro--Viana, we call
\(\boldsymbol\eta\) a \emph{unit vector} when every component
\(\eta_i\) is a probability measure.

Let $\boldsymbol\pi(P)=(\pi_1,\ldots,\pi_N)$
be the stationary probability vector of \(P\).  Define the operator \(\mathcal P_{\mathbf A,P}\) by
\begin{equation}\label{eq:projective-measure-vector-operator}
 \bigl(\mathcal P_{\mathbf A,P}\boldsymbol\eta\bigr)_j
 =
 \sum_{i\in X}
 \frac{\pi_i p_{ij}}{\pi_j}\,(F_i)_*\eta_i.
\end{equation}
The coefficients
\[
 C_{ji}
 =
 \frac{\pi_i p_{ij}}{\pi_j}
\]
are the transition probabilities of the reversed stationary Markov
chain.  Indeed,
\[
 \sum_{i\in X}C_{ji}
 =
 \frac{1}{\pi_j}\sum_{i\in X}\pi_i p_{ij}
 =
 \frac{\pi_j}{\pi_j}
 =
 1.
\]
It follows that \(\mathcal P_{\mathbf A,P}\) maps unit vectors to unit
vectors.

We say that \(\boldsymbol\eta\) is
\emph{\(\mathcal P_{\mathbf A,P}\)-stationary} when
$
 \mathcal P_{\mathbf A,P}\boldsymbol\eta
 =
 \boldsymbol\eta.
$
Equivalently, for every \(j\in X\) and every
\(\varphi\in C(\mathcal Z)\),
\begin{align}
 \int_{\mathcal Z}\varphi\,d\eta_j
 &=
 \sum_{i\in X}
 \frac{\pi_i p_{ij}}{\pi_j}
 \int_{\mathcal Z}\varphi\circ F_i\,d\eta_i
 \notag\\
 &=
 \sum_{i\in X}
 \frac{\pi_i p_{ij}}{\pi_j}
 \int_{\mathcal Z}
 \varphi\bigl(t+\theta_i,[A_i(t)v]\bigr)
 \,d\eta_i(t,[v]).
 \label{eq:projective-stationary-vector-integral}
\end{align}



Define the continuous Markov operator \(\mathcal Q\) on
$
 Y=X\times\T^m\times\PP^{d-1}
$
whose transition kernel is
\begin{equation}\label{eq:projective-Markov-kernel}
 \mathcal Q\bigl((i,t,[v]),\cdot\bigr)
 =
 \sum_{j\in X}p_{ij}\,
 \delta_{\left(
 j,\,
 t+\theta_i,\,
 [A_i(t)v]
 \right)}.
\end{equation}
To a unit vector \(\boldsymbol\eta\), associate the probability
measure
\begin{equation}\label{eq:kernel-measure-from-vector}
 \overline\nu_{\boldsymbol\eta}
 =
 \sum_{i\in X}\pi_i\,\delta_i\times\eta_i
\end{equation}
on \(Y\).  A direct calculation shows that
\begin{equation}\label{eq:kernel-vector-correspondence}
 \mathcal Q_*\overline\nu_{\boldsymbol\eta}
 =
 \overline\nu_{
 \mathcal P_{\mathbf A,P}\boldsymbol\eta}.
\end{equation}
Therefore
\[
 \boldsymbol\eta
 \text{ is \(\mathcal P_{\mathbf A,P}\)-stationary}
 \quad\Longleftrightarrow\quad
 \overline\nu_{\boldsymbol\eta}
 \text{ is \(\mathcal Q\)-stationary}.
\]

Conversely, let \(\nu\) be an arbitrary \(\mathcal Q\)-stationary
probability measure on \(Y\).  Its projection onto \(X\) is stationary
for \(P\).  Since \(P\) is irreducible, this projection is necessarily
\(\boldsymbol\pi(P)\).  Hence \(\nu\) has a unique disintegration of
the form
\[
 \nu
 =
 \sum_{i\in X}\pi_i\,\delta_i\times\eta_i
\]
for some unit vector
\(\boldsymbol\eta=(\eta_i)_{i\in X}\).  The stationarity of \(\nu\)
and \eqref{eq:kernel-vector-correspondence} then imply
\[
 \mathcal P_{\mathbf A,P}\boldsymbol\eta
 =
 \boldsymbol\eta.
\]
Thus stationary probabilities for \(\mathcal Q\) are in one-to-one
correspondence with stationary projective unit vectors.

We now identify the torus marginal of such a vector.  Let
\[
 \eta_i^0
 =
 \bigl(\operatorname{proj}_{\T^m}\bigr)_*\eta_i.
\]
Projecting \eqref{eq:projective-stationary-vector-integral} onto the
torus coordinate gives
\[
 \eta_j^0
 =
 \sum_{i\in X}
 \frac{\pi_i p_{ij}}{\pi_j}
 (f_i)_*\eta_i^0,
 \qquad
 f_i(t)=t+\theta_i.
\]
Thus
\[
 \boldsymbol\eta^0=(\eta_i^0)_{i\in X}
\]
is a stationary unit vector for the torus measure vector operator
\(\mathcal P_P\).

By Proposition~\ref{prop:unique-torus-vector}, that stationary torus
unit vector is unique and is given by
\[
 \boldsymbol\eta^0
 =
 (\Leb,\ldots,\Leb).
\]
Consequently,
\begin{equation}\label{eq:projective-vector-torus-marginal}
 \bigl(\operatorname{proj}_{\T^m}\bigr)_*\eta_i
 =
 \Leb
 \qquad\text{for every }i\in X.
\end{equation}
In particular, every component \(\eta_i\) admits a disintegration
over Lebesgue measure:
\begin{equation}\label{eq:projective-vector-disintegration}
 d\eta_i(t,[v])
 =
 d\Leb(t)\,d\eta_{i,t}([v]),
\end{equation}
where
\[
 t\longmapsto\eta_{i,t}\in\Prob(\PP^{d-1})
\]
is a measurable family, defined for Lebesgue almost every \(t\).

Substituting \eqref{eq:projective-vector-disintegration} into the
stationarity equation gives, for every \(j\in X\) and
Lebesgue-almost every \(s\in\T^m\),
\begin{equation}\label{eq:conditional-projective-vector-equation}
 \eta_{j,s}
 =
 \sum_{i\in X}
 \frac{\pi_i p_{ij}}{\pi_j}
 \bigl(A_i(s-\theta_i)\bigr)_*
 \eta_{i,s-\theta_i}.
\end{equation}
This is the projective, torus dependent version of the stationary
measure vector equation of Malheiro and Viana.

Define $g:Y\longrightarrow\R$ by
\begin{equation}\label{eq:projective-logarithmic-observable}
 g(i,t,[v])
 =
 \log
 \frac{\|A_i(t)v\|}{\|v\|},
\end{equation}
we have, 
\[
 \|A_i(t)^{-1}\|^{-1}
 \leq
 \frac{\|A_i(t)v\|}{\|v\|}
 \leq
 \|A_i(t)\|.
\]
Therefore, if
\[
 C_0
 =
 \max_{\substack{i\in X\\t\in\T^m}}
 \max\left\{
 \log\|A_i(t)\|,
 \log\|A_i(t)^{-1}\|
 \right\},
\]
then
\[
 |g(i,t, [v] )|\leq C_0.
\]
The function \(g\) is continuous because the matrices \(A_i(t)\)
depend continuously on \(t\).

For a projective unit vector \(\boldsymbol\eta\), define 
\begin{align}
 \alpha(\boldsymbol\eta)
 &:=
 \int_Y g\,d\overline\nu_{\boldsymbol\eta}
 \notag\\
 &=
 \sum_{i\in X}\pi_i
 \int_{\T^m\times\PP^{d-1}}
 \log\frac{\|A_i(t)v\|}{\|v\|}
 \,d\eta_i(t, [v]).
 \label{eq:Furstenberg-functional-vector}
\end{align}
If \(\boldsymbol\eta\) is stationary, then
\(\overline\nu_{\boldsymbol\eta}\) is stationary for \(\mathcal Q\).
The cocycle identity gives
\[
 \log
 \frac{\|A_x^n(t)v\|}{\|v\|}
 =
 \sum_{r=0}^{n-1}
 g\left(
 x_r,\,
 t_r,\,
 [A_x^r(t)v]
 \right),
\]
where
\[
 t_r=t+\sum_{q=0}^{r-1}\theta_{x_q}.
\]
Integrating this identity with respect to the corresponding invariant
measure gives
\begin{equation}\label{eq:integrated-n-step-projective-expansion}
 \int
 \log\frac{\|A_x^n(t)v\|}{\|v\|}
 \,d\bigl(\mu_P\ltimes\boldsymbol\eta\bigr)
 =
 n\alpha(\boldsymbol\eta).
\end{equation}

The Furstenberg and Kifer formula in the Markov setting
states that
\begin{equation}\label{eq:Furstenberg-formula-projective-vectors}
 \lambda_1(P)
 =
 \max\left\{
 \alpha(\boldsymbol\eta):
 \boldsymbol\eta
 \text{ is a stationary projective unit vector}
 \right\}.
\end{equation}
The maximum formula, by itself, does not say that
every stationary projective unit vector realizes the top exponent. To obtain that, we use the non-random filtration of Furstenberg and Kifer.  If
\(\boldsymbol\eta\) is an extremal stationary projective unit vector,
then the associated stationary measure
\(\overline\nu_{\boldsymbol\eta}\) is ergodic for \(\mathcal Q\), and
\(\alpha(\boldsymbol\eta)\) is one of the exponents from the
non-random filtration.  If
\[
 \alpha(\boldsymbol\eta)<\lambda_1(P),
\]
the filtration produces a proper measurable family of subspaces
\[
 \mathcal L_i(t)\subsetneq\R^d
\]
satisfying, along every admissible edge \(i\to j\),
\begin{equation}\label{eq:lower-filtration-invariant-family}
 A_i(t)\mathcal L_i(t)
 =
 \mathcal L_j(t+\theta_i)
\end{equation}
for Lebesgue-almost every \(t\).

The irreducibility hypothesis excludes every such proper invariant
family.  Consequently, every extremal stationary projective unit
vector satisfies
\[
 \alpha(\boldsymbol\eta)=\lambda_1(P).
\]
An arbitrary stationary unit vector is a convex combination of extremal
stationary unit vectors.  Since
\(\alpha\) is affine, the same equality holds for every stationary
projective unit vector:
\begin{equation}\label{eq:every-stationary-vector-realizes-top}
 \alpha(\boldsymbol\eta)=\lambda_1(P)
 \qquad
 \text{for every stationary projective unit vector }
 \boldsymbol\eta.
\end{equation}
Equivalently, every \(\mathcal Q\)-stationary probability measure
\(\nu\) satisfies
\begin{equation}\label{eq:every-stationary-measure-realizes-top}
 \int_Y g\,d\nu=\lambda_1(P).
\end{equation}

\medskip
\noindent\emph{Uniform convergence.}
For \(i\in X\), \(t\in\T^m\), and \([v]\in\PP^{d-1}\), define
\begin{equation}\label{eq:averaged-finite-time-expansion}
 L_n(i,t,[v])
 =
 \frac1n
 \int_{\Sigma_i}
 \log\frac{\|A_x^n(t)v\|}{\|v\|}
 \,d\mu_{P,i}(x),
\end{equation}
where \(v\neq0\) is any representative of \([v]\).

We claim that
\begin{equation}\label{eq:uniform-averaged-growth}
 \lim_{n\to\infty}
 \sup_{\substack{i\in X,\;t\in\T^m\\
                 [v]\in\PP^{d-1}}}
 \left|
 L_n(i,t,[v])-\lambda_1(P)
 \right|
 =
 0.
\end{equation}

Fix $z=(i,t,[v])\in Y$. The cocycle identity gives
\begin{equation}\label{eq:additive-cocycle-identity}
 \log\frac{\|A_x^n(t)v\|}{\|v\|}
 =
 \sum_{\ell=0}^{n-1}
 g\left(
 x_\ell,\,
 t_\ell,\,
 [A_x^\ell(t)v]
 \right).
\end{equation}
The distribution of $\left(
 x_\ell,\,
 t_\ell,\,
 [A_x^\ell(t)v]
 \right)$
under \(\mu_{P,i}\) is precisely
\(\mathcal Q^\ell_*\delta_z\).  Integrating
\eqref{eq:additive-cocycle-identity} therefore yields
\begin{align}
 L_n(i,t,\bar v)
 &=
 \frac1n
 \sum_{\ell=0}^{n-1}
 \int_Yg\,d\bigl(\mathcal Q^\ell_*\delta_z\bigr).
 \label{eq:Ln-kernel-iterates}
\end{align}
Let
\begin{equation}\label{eq:projective-occupation-measure}
 \nu_{n,z}
 =
 \frac1n
 \sum_{\ell=0}^{n-1}
 \mathcal Q^\ell_*\delta_z.
\end{equation}
Then
\begin{equation}\label{eq:Ln-occupation-integral}
 L_n(i,t,[v])
 =
 \int_Yg\,d\nu_{n,z}.
\end{equation}

We have that, 
\begin{align}
 \mathcal Q_*\nu_{n,z}-\nu_{n,z}
 &=
 \frac1n
 \sum_{\ell=0}^{n-1}
 \left(
 \mathcal Q^{\ell+1}_*\delta_z
 -
 \mathcal Q^\ell_*\delta_z
 \right)
 \notag\\
 &=
 \frac1n
 \left(
 \mathcal Q^n_*\delta_z-\delta_z
 \right).
 \label{eq:occupation-stationarity-defect}
\end{align}
Hence, for every \(\varphi\in C(Y)\),
\begin{align}
 \left|
 \int_Y\varphi\,d(\mathcal Q_*\nu_{n,z})
 -
 \int_Y\varphi\,d\nu_{n,z}
 \right|
 &=
 \frac1n
 \left|
 \int_Y\varphi\,d(\mathcal Q^n_*\delta_z)
 -
 \varphi(z)
 \right|
 \notag\\
 &\leq
 \frac{2\|\varphi\|_\infty}{n}.
 \label{eq:uniform-stationarity-defect}
\end{align}
This estimate is independent of the initial point \(z\).

Suppose, by contradiction, that
\eqref{eq:uniform-averaged-growth} fails.  Then there exist
\(\varepsilon>0\), integers \(n_r\to\infty\), and points $ z_r=(i_r,t_r,[v_r])\in Y$ such that
\begin{equation}\label{eq:failure-of-uniform-growth}
 \left|
 L_{n_r}(i_r,t_r,[v_r])-\lambda_1(P)
 \right|
 \geq\varepsilon
 \qquad\text{for every }r.
\end{equation}

The phase space \(Y\) is compact.  Therefore \(\Prob(Y)\) is
weak* compact, and, after passing to a subsequence, we may assume
that $ \nu_{n_r,z_r}\longrightarrow\nu$ weakly for some \(\nu\in\Prob(Y)\).

We first show that \(\nu\) is stationary for \(\mathcal Q\).  Let
\(\varphi\in C(Y)\). 
\[
 (\mathcal Q\varphi)(z)
 =
 \int_Y\varphi(w)\,d\mathcal Q(z,w)
\]
is continuous.  It follows that
\begin{align*}
 \int_Y\varphi\,d(\mathcal Q_*\nu)
 -
 \int_Y\varphi\,d\nu
 &=
 \int_Y\mathcal Q\varphi\,d\nu
 -
 \int_Y\varphi\,d\nu
 \\
 &=
 \lim_{r\to\infty}
 \left[
 \int_Y\mathcal Q\varphi\,d\nu_{n_r,z_r}
 -
 \int_Y\varphi\,d\nu_{n_r,z_r}
 \right].
\end{align*}
By \eqref{eq:uniform-stationarity-defect}, the absolute value of the
expression inside the limit is at most
$ 2\|\varphi\|_\infty/{n_r},
$
which converges to zero.  Thus
\[
 \int_Y\varphi\,d(\mathcal Q_*\nu)
 =
 \int_Y\varphi\,d\nu
 \qquad
 \text{for every }\varphi\in C(Y),
\]
and consequently
\[
 \mathcal Q_*\nu=\nu.
\]

By the correspondence established above, there exists a
stationary projective unit vector
\(\boldsymbol\eta=(\eta_i)_{i\in X}\) such that
\[
 \nu
 =
 \sum_{i\in X}\pi_i\,\delta_i\times\eta_i.
\]
The previously proved uniqueness of the stationary torus unit vector
ensures that
\[
 \bigl(\operatorname{proj}_{\T^m}\bigr)_*\eta_i
 =
 \Leb
 \qquad\text{for every }i.
\]
Since the cocycle is irreducible, Step~1 gives
\[
 \int_Yg\,d\nu
 =
 \lambda_1(P).
\]

Finally, \(g\) is continuous and
\(\nu_{n_r,z_r}\to\nu\) weakly.  Therefore,
using \eqref{eq:Ln-occupation-integral},
$$
 \lim_{r\to\infty}
 L_{n_r}(i_r,t_r,[v_r])
 =
 \lim_{r\to\infty}
 \int_Yg\,d\nu_{n_r,z_r} = 
 \int_Yg\,d\nu =  \lambda_1(P).
$$
This contradicts
\eqref{eq:failure-of-uniform-growth}.  Hence
\[
 \sup_{\substack{i\in X,\;t\in\T^m\\
                 [v]\in\PP^{d-1}}}
 \left|
 \frac1n
 \int_{\Sigma_i}
 \log\frac{\|A_x^n(t)v\|}{\|v\|}
 \,d\mu_{P,i}(x)
 -
 \lambda_1(P)
 \right|
 \longrightarrow0.
\]
Hence uniform convergence.

\noindent\emph{The uniform exterior squar upper bound.}
Consider
\[
 h_n(i,t)=
 \int_{\Sigma_i}
 \log\norm{{\textstyle\bigwedge^2}A_x^n(t)}
 \,d\mu_{P,i}(x).
\]
These continuous functions satisfy the Markov subadditive inequality
\[
 h_{n+\ell}\le h_n+\mathcal K_P^nh_\ell,
\]
where \(\mathcal K_P\) is the base Markov operator
\[
 (\mathcal K_Ph)(i,t)
 =\sum_jp_{ij}h(j,t+\theta_i).
\]
The compact base
has the unique stationary probability given in
Proposition~\ref{prop:unique-torus-vector}.  The standard blocking
argument for a subadditive sequence yields
\[
 \limsup_{n\to\infty}\frac1n\sup_{i,t}h_n(i,t)
 \le
 \inf_{q\ge1}\frac1q
 \sum_i\pi_i\int_{\T^m}h_q(i,t)\,dt.
\]

Kingman's theorem identifies that infimum with
\[
 \lambda_+\!\left({\textstyle\bigwedge^2}\mathbf A,P\right)
 =\lambda_1(P)+\lambda_2(P).
\]
This proves the required uniform upper estimate.
\end{proof}


Now let
\[
 M=\max_{i,t}\norm{A_i(t)},\qquad
 m_*=\min_{i,t}\norm{A_i(t)^{-1}}^{-1},\qquad
 \kappa=M/m_*.
\]

\begin{lemma}[Uniform distortion bound]\label{lem:distortion}
For every \(n\ge1\), every path \(x\), every \(t\), and every
distinct \(u,v\in\PP^{d-1}\),
\[
 \kappa^{-2n}
 \le
 \frac{d(A_x^n(t)u,A_x^n(t)v)}{d(u,v)}
 \le
 \kappa^{2n}.
\]
\end{lemma}

\begin{proof}
Let \(B\in\operatorname{GL}_d(\R)\).  For non-zero representatives
of \(u,v\),
\[
\begin{aligned}
 d(Bu,Bv)
 &=\frac{\norm{Bu\wedge Bv}}{\norm{Bu}\norm{Bv}}\\
 &\le
 \frac{\norm{\bigwedge^2B}\,\norm{u\wedge v}}
      {\norm{B^{-1}}^{-2}\norm u\,\norm v}\\
 &\le \norm B^2\norm{B^{-1}}^2d(u,v).
\end{aligned}
\]
Applying this inequality to \(B^{-1}\), with \(Bu,Bv\) in place of
\(u,v\), gives the reverse bound
\[
 d(Bu,Bv)\ge
 \bigl(\norm B\norm{B^{-1}}\bigr)^{-2}d(u,v).
\]
Every one-step matrix has norm at most \(M\) and least singular value
at least \(m_*\).  Hence
\[
 \norm{A_x^n(t)}
 \norm{A_x^n(t)^{-1}}
 \le(M/m_*)^n=\kappa^n.
\]
Substitution proves both inequalities.
\end{proof}

Now let
\[
 \psi_n(x,t;u,v)
 =\frac{d(A_x^n(t)u,A_x^n(t)v)}{d(u,v)}
\]
Lemma~\ref{lem:distortion} gives
\begin{equation}\label{eq:logbound}
 |\log\psi_n(x,t;u,v)|\le2n\log\kappa.
\end{equation}
If \(\kappa=1\), every \(A_i(t)\) is conformal up to scale and no
simple top exponent exists when \(d>1\).  Thus the interesting
projective case has \(\kappa>1\).

\begin{proposition}[Projective contraction]\label{prop:contraction}
Under the hypotheses of Theorem~\ref{thm:irreducible-core}, there are
\(\alpha_*>0\), \(n_*\ge1\), and \(c_*>0\) such that, for every fixed
\(\alpha\in(0,\alpha_*]\),
\[
 \widehat K_n(\alpha,P)
 \le C(\alpha)e^{-\zeta(\alpha)n},
 \qquad n\ge0,
\]
where one may take
\[
 \zeta(\alpha)
 =-\frac1{n_*}\log\widehat K_{n_*}(\alpha,P)>0.
\]
For small \(\alpha\), the construction gives
\(\zeta(\alpha)\ge c_*\alpha\gap\).
\end{proposition}

\begin{proof}
Write \(\Delta=\gap =\lambda_1(P) - \lambda_2(P)>0\). 

\smallskip

Lemma~\ref{lem:uniform-growth} gives
\[
 \frac1n\int_{\Sigma_i}
 \log\norm{A_x^n(t)v}\,d\mu_{P,i}(x)
 \longrightarrow\lambda_1(P)
\]
uniformly in \(i,t\), and \(v\in\PP^{d-1}\).  Moreover,
\[
 \limsup_{n\to\infty}\max_i\sup_t\frac1n
 \int_{\Sigma_i}
 \log\norm{\textstyle\bigwedge^2 A_x^n(t)}
 \,d\mu_{P,i}(x)
 \le\lambda_1(P)+\lambda_2(P).
\]
Choose \(\varepsilon=\Delta/8\).  By the definition of uniform
convergence and of the uniform limsup, there exists \(N_\varepsilon\)
such that every \(n\ge N_\varepsilon\) satisfies, simultaneously for
all \(i,t,u,v\),
\[
 \frac1n\int_{\Sigma_i}
 \log\norm{{\textstyle\bigwedge^2}A_x^{n}(t)}
 \,d\mu_{P,i}(x)
 \le\lambda_1+\lambda_2+\varepsilon
\]
and
\[
 \frac1n\int_{\Sigma_i}
 \log\norm{A_x^{n}(t)u}\,d\mu_{P,i}(x)
 \ge\lambda_1-\varepsilon,\qquad
 \frac1n\int_{\Sigma_i}
 \log\norm{A_x^{n}(t)v}\,d\mu_{P,i}(x)
 \ge\lambda_1-\varepsilon.
\]
Fix once and for all one \(n_*\ge N_\varepsilon\).
Consequently,
\[
 \max_i\sup_t\sup_{u\ne v}
 \int_{\Sigma_i}\log\psi_{n_*}(x,t;u,v)
 \,d\mu_{P,i}(x)
 \le-\frac58n_*\Delta
 \le-\frac12n_*\Delta.
\]

\smallskip
Using
\[
 e^s\le1+s+\frac12s^2e^{|s|}
\]
and \eqref{eq:logbound}, and setting
\[
 X=\log\psi_{n_*}(x,t;u,v),\qquad
 B_*=2n_*\log\kappa.
\]
Then \(|X|\le B_*\), and uniformly in \(i,t,u,v\),
\[
\begin{aligned}
 \int_{\Sigma_i}e^{\alpha X}\,d\mu_{P,i}
 &\le 1+\alpha\int_{\Sigma_i}X\,d\mu_{P,i}
       +\frac{\alpha^2}{2}
        \int_{\Sigma_i}X^2e^{\alpha|X|}\,d\mu_{P,i}\\
 &\le 1-\frac12\alpha n_*\gap
       +\frac12\alpha^2B_*^2e^{\alpha B_*}.
\end{aligned}
\]
The function
\[
 a\longmapsto aB_*^2e^{aB_*}
\]
is continuous and vanishes at \(a=0\).  Hence one can choose
\(\alpha_*>0\) so small that both
\[
 \alpha_*B_*^2e^{\alpha_*B_*}
 \le\frac12n_*\Delta
\]
and
\[
 \frac14\alpha_*n_*\Delta<1
\]
hold.  For 
\(0<\alpha\le\alpha_*\), 
$\frac12\alpha^2B_*^2e^{\alpha B_*} \le\frac14\alpha n_*\Delta$.
This gives
\[
 \widehat K_{n_*}(\alpha,P)
 \le
 1-\frac14\alpha n_*\Delta<1.
\]
This Taylor estimate is used only at the fixed block length \(n_*\).

\smallskip
\noindent\emph{Iterate the good block.}
Set
\[
 a_\alpha=\widehat K_{n_*}(\alpha,P),\qquad
 C_{\mathrm{rem}}(\alpha)
 =\max_{0\le r<n_*}\widehat K_r(\alpha,P).
\]
Every remainder coefficient is finite by
Lemma~\ref{lem:distortion}; in fact
\[
 C_{\mathrm{rem}}(\alpha)
 \le\max_{0\le r<n_*}\kappa^{2\alpha r}.
\]
Writing \(n=qn_*+r\), \(0\le r<n_*\), gives
\[
 \widehat K_n(\alpha,P)
 \le
 a_\alpha^qC_{\mathrm{rem}}(\alpha)
\]
by Lemma~\ref{lem:submult}.  Define
\[
 \zeta(\alpha)=-\frac1{n_*}\log a_\alpha>0.
\]
Since \(q\ge n/n_*-1\),
\[
 a_\alpha^q
 =e^{-qn_*\zeta(\alpha)}
 \le e^{n_*\zeta(\alpha)}e^{-n\zeta(\alpha)}
 =a_\alpha^{-1}e^{-n\zeta(\alpha)}.
\]
Thus one may take
\[
 C(\alpha)=a_\alpha^{-1}C_{\mathrm{rem}}(\alpha).
\]
Finally, \(-\log(1-s)\ge s\), with
\(s=\alpha n_*\Delta/4\), shows that
\[
 \zeta(\alpha)
 =-\frac1{n_*}\log a_\alpha
 \ge\frac14\alpha\Delta.
\]
\end{proof}

\section{The Markov transfer operator}
Fix \(0<\alpha\leq1\).  Let
\[
 \mathcal B_\alpha
 =
 \left\{
 \Phi\in
 C\bigl(X\times\T^m\times\PP^{d-1};\C\bigr):
 [\Phi]_\alpha<\infty
 \right\},
\]
where
\begin{equation}\label{eq:projective-holder-seminorm}
 [\Phi]_\alpha
 =
 \max_{i\in X}
 \sup_{t\in\T^m}
 \sup_{\substack{u, v\in\PP^{d-1}\\
                  u\neq v}}
 \frac{
 |\Phi(i,t,u)-\Phi(i,t,v)|
 }{
 d(u,v)^\alpha
 }.
\end{equation}
Thus functions in \(\mathcal B_\alpha\) are continuous in all
variables and uniformly \(\alpha\)-Hölder in the projective
coordinate.  We equip \(\mathcal B_\alpha\) with the norm
\begin{equation}\label{eq:projective-holder-norm}
 \|\Phi\|_{\mathcal B_\alpha}
 =
 \|\Phi\|_\infty+[\Phi]_\alpha.
\end{equation}
With this norm, \(\mathcal B_\alpha\) is a complex Banach space.

Let
\[
 E(U)
 =
 \left\{
 Z=(z_{ij})\in\C^{N\times N}:
 z_{ij}=0\text{ whenever }u_{ij}=0
 \right\}.
\]
This is the complex vector space of matrices supported on \(U\).
The complex affine space of row-stochastic perturbations is
\[
 H(U)
 =
 \left\{
 Z\in E(U):
 \sum_{j=1}^Nz_{ij}=1
 \text{ for every }i
 \right\}.
\]

For \(Z\in E(U)\), define the forward operator
\(\mathcal T_Z:\mathcal B_\alpha\to\mathcal B_\alpha\) by
\begin{equation}\label{eq:forward-Markov-operator}
 (\mathcal T_Z\Phi)(i,t, v)
 =
 \sum_{j=1}^Nz_{ij}\,
 \Phi\left(
 j,\,
 t+\theta_i,\,
 A_i(t)v
 \right),
\end{equation}
%
\(i\) is the present state, so the matrix and translation used
during the current step are \(A_i(t)\) and \(\theta_i\), while \(j\)
is the state reached after that step.

For \(i\in X\) and \(n\geq1\), let \(P_n(i)\) denote the
set of admissible paths
\[
 \mathbf i=(i_0,i_1,\ldots,i_n)
\]
of length \(n\) such that \(i_0=i\) and
\[
 u_{i_ri_{r+1}}=1,
 \qquad 0\leq r\leq n-1.
\]
For such a path, define its complex weight by
\begin{equation}\label{eq:complex-path-weight}
 w_Z(\mathbf i)
 =
 \prod_{r=0}^{n-1}z_{i_ri_{r+1}}.
\end{equation}
Starting from \(t_0=t\), define inductively $t_{r+1}=t_r+\theta_{i_r}$.
Equivalently,
\begin{equation}\label{eq:path-torus-points}
 t_r
 =
 t+\sum_{s=0}^{r-1}\theta_{i_s},
 \qquad 1\leq r\leq n.
\end{equation}
The matrix product associated with \(\mathbf i\) is
\begin{equation}\label{eq:path-matrix-product}
 A_{\mathbf i}^n(t)
 =
 A_{i_{n-1}}(t_{n-1})
 \cdots
 A_{i_1}(t_1)
 A_{i_0}(t_0).
\end{equation}

An induction using \eqref{eq:forward-Markov-operator} gives
\begin{equation}\label{eq:iterate-forward-Markov-operator}
 (\mathcal T_Z^n\Phi)(i,t, v)
 =
 \sum_{\mathbf i\in P_n(i)}
 w_Z(\mathbf i)\,
 \Phi\left(
 i_n,\,
 t_n,\,
 A_{\mathbf i}^n(t)v
 \right).
\end{equation}
For fixed \(n\) and \(\Phi\), the right-hand side is a
\(\mathcal B_\alpha\)-valued homogeneous polynomial of degree \(n\)
in the entries of \(Z\in E(U)\).  Its restriction to the affine
subspace \(H(U)\) is therefore a holomorphic polynomial of degree at
most \(n\).

\begin{lemma}\label{lem:operator-bounded}
For every \(Z\in E(U)\), the operator
\[
 \mathcal T_Z:\mathcal B_\alpha\longrightarrow\mathcal B_\alpha
\]
is bounded.  More precisely, if
\[
 M(Z)=\max_{i\in X}\sum_{j=1}^N|z_{ij}|
\]
then
\begin{equation}\label{eq:operator-norm-bound}
 \|\mathcal T_Z\|_{\mathcal L(\mathcal B_\alpha)}
 \leq
 M(Z)\max\{1,\kappa^{2\alpha}\}.
\end{equation}
Consequently, if \(V\subset E(U)\) is bounded, then
\[
 \sup_{Z\in V}
 \|\mathcal T_Z\|_{\mathcal L(\mathcal B_\alpha)}
 <\infty.
\]
Moreover,
\[
 E(U)\ni Z\longmapsto
 \mathcal T_Z\in\mathcal L(\mathcal B_\alpha)
\]
is complex-linear and hence entire.  In particular, its restriction
to \(H(U)\) is holomorphic.
\end{lemma}

\begin{proof}
First observe that \(\mathcal T_Z\Phi\) is continuous.  This follows
from the continuity of \(\Phi\), the continuity of \(t\mapsto A_i(t)\),
and the fact that the sum in
\eqref{eq:forward-Markov-operator} is finite.

For the uniform norm, we have
\begin{align}
 |(\mathcal T_Z\Phi)(i,t, v)|
 &\leq
 \sum_{j=1}^N|z_{ij}|
 \left|
 \Phi\left(
 j,t+\theta_i,A_i(t)v
 \right)
 \right|
 \notag\\
 &\leq
 \left(\sum_{j=1}^N|z_{ij}|\right)
 \|\Phi\|_\infty.
\end{align}
Taking the supremum over \(i,t, v\) gives
\begin{equation}\label{eq:operator-sup-bound}
 \|\mathcal T_Z\Phi\|_\infty
 \leq
 M(Z)\|\Phi\|_\infty.
\end{equation}

We next estimate the projective Hölder seminorm.
Let \( u\neq  v\).  We have that
\begin{align*}
 &\left|
 (\mathcal T_Z\Phi)(i,t,u)
 -
 (\mathcal T_Z\Phi)(i,t, v)
 \right|
 \\
 &\quad\leq
 \sum_{j=1}^N|z_{ij}|
 \left|
 \Phi\left(
 j,t+\theta_i, A_i(t)u
 \right)
 -
 \Phi\left(
 j,t+\theta_i,A_i(t)v
 \right)
 \right|
 \\
 &\quad\leq
 \left(\sum_{j=1}^N|z_{ij}|\right)
 [\Phi]_\alpha
 d\left( A_i(t)u, A_i(t)v
 \right)^\alpha
 \\
 &\quad\leq
 M(Z)\kappa^{2\alpha}
 [\Phi]_\alpha
 d(u, v)^\alpha.
\end{align*}
It follows that
\begin{equation}\label{eq:operator-holder-bound}
 [\mathcal T_Z\Phi]_\alpha
 \leq
 M(Z)\kappa^{2\alpha}[\Phi]_\alpha.
\end{equation}
Combining \eqref{eq:operator-sup-bound} and
\eqref{eq:operator-holder-bound}, we get
\begin{align*}
 \|\mathcal T_Z\Phi\|_{\mathcal B_\alpha}
 &\leq
 M(Z)\|\Phi\|_\infty
 +
 M(Z)\kappa^{2\alpha}[\Phi]_\alpha\\
 &\leq
 M(Z)\max\{1,\kappa^{2\alpha}\}
 \|\Phi\|_{\mathcal B_\alpha}.
\end{align*}
This proves \eqref{eq:operator-norm-bound}.

If \(V\subset E(U)\) is bounded, then, because \(E(U)\) is finite
dimensional,
\[
 M_V:=\sup_{Z\in V}M(Z)<\infty.
\]
Hence
\[
 \sup_{Z\in V}
 \|\mathcal T_Z\|_{\mathcal L(\mathcal B_\alpha)}
 \leq
 M_V\max\{1,\kappa^{2\alpha}\}<\infty.
\]

Finally, for each admissible edge \((i,j)\in\mathcal E\), define the
bounded operator \(\mathcal L_{ij}\) by
\[
 (\mathcal L_{ij}\Phi)(r,t, v)
 =
 \mathbf 1_{\{i\}}(r)\,
 \Phi\left(
 j,t+\theta_i,A_i(t)v \right).
\]
Then
\begin{equation}\label{eq:operator-linear-decomposition}
 \mathcal T_Z
 =
 \sum_{(i,j)\in\mathcal E}z_{ij}\mathcal L_{ij}.
\end{equation}
This is a finite sum of bounded operators and shows that
\[
 Z\longmapsto\mathcal T_Z
\]
is complex-linear from \(E(U)\) into
\(\mathcal L(\mathcal B_\alpha)\).  In particular, it is entire in
the operator norm.  Its restriction to the affine subspace \(H(U)\)
is therefore holomorphic.
\end{proof}

Now we fix \(\alpha\in(0,\alpha_*]\) and let 
\(\zeta=\zeta(\alpha)\).  Choose
\[
 e^{-\zeta}<\gamma<1,
\]
and define
\[
 D_\gamma
 =\left\{Z\in H(U):
       \max_{(i,j)\in\cE}\frac{|z_{ij}|}{p_{ij}}
       <\gamma^{-1}\right\}.
\]
The maximum is taken only over the support, so no division by zero
occurs.  For every admissible path,
\begin{align}
 |w_Z(\mathbf i)|
 &=\prod_{r=0}^{n-1}|z_{i_ri_{r+1}}|\notag\\
 &\le
 \prod_{r=0}^{n-1}
 \bigl(\gamma^{-1}p_{i_ri_{r+1}}\bigr)\notag\\
 &=\gamma^{-n}w_P(\mathbf i).
\label{eq:path-comparison}
\end{align}

\section{Holomorphic extension}
\label{sec:direct}

We now prove analyticity. The argument has
three ingredients: the normalized complex left eigenvector, polynomial
dependence of finite iterates, and a summable estimate for consecutive
logarithmic increments.

\subsection{The normalized complex left eigenvector}

A complex matrix does not have a stationary probability distribution.
Nevertheless, because every \(Z\in H(U)\) satisfies
\[
 Z\mathbf{e}=\mathbf{e},
\]
where $\mathbf{e}:=(1,\ldots,1)^T\in\C^N$,  the eigenvalue \(1\) remains present after complexification.

\begin{lemma}\label{lem:direct-left-vector}
There are a complex neighborhood \(V\subset H(U)\) of \(P\) and a
unique holomorphic map
\[
 Z\longmapsto
 \boldsymbol\ell(Z)=(\ell_1(Z),\ldots,\ell_N(Z))
\]
such that
\[
 \boldsymbol\ell(Z)Z=\boldsymbol\ell(Z),
 \qquad
 \sum_i\ell_i(Z)=1.
\]
For every real \(Q\in V\cap S(U)\),
\(\boldsymbol\ell(Q)=\boldsymbol\pi(Q)\), where $\boldsymbol{\pi}(Q)$ is the stationary distrobution of $Q$.
\end{lemma}

\begin{proof}
Define
$$
 \mathcal R:\C^N\longrightarrow\C^N,
 \qquad
 \mathcal Rv
 :=
 \left(\sum_{i=1}^N v_i\right)\mathbf e.
$$
Thus, \(\mathcal Rv\) is the constant vector whose entries are all
equal to the sum of the coordinates of \(v\).

For \(Z\in H(U)\), define $B(Z):=I-Z^T+\mathcal R$.
Since \(P\) is irreducible, the eigenvalue \(1\) of \(P\) is simple. Hence $B(P)$ is invertible
%
After shrinking the complex neighborhood \(V\) of \(P\), the matrix
\(B(Z)\) remains invertible for every \(Z\in V\).  Define
\[
 \boldsymbol\ell(Z)^T:=B(Z)^{-1}\mathbf e.
\]
Equivalently,
\begin{equation}\label{eq:normalized-left-vector-system}
 (I-Z^T)\boldsymbol\ell(Z)^T
 +
 \left(\sum_{i=1}^N\ell_i(Z)\right)\mathbf e
 =
 \mathbf e.
\end{equation}

Since \(Z\mathbf e=\mathbf e\), we have
\[
 \mathbf e^T(I-Z^T)=0.
\]
Multiplying \eqref{eq:normalized-left-vector-system} on the left by
\(\mathbf e^T\) gives
\[
 N\sum_{i=1}^N\ell_i(Z)=N.
\]
Therefore
\[
 \sum_{i=1}^N\ell_i(Z)=1.
\]
Substituting this identity back into
\eqref{eq:normalized-left-vector-system} yields
\[
 (I-Z^T)\boldsymbol\ell(Z)^T=0.
\]
Equivalently,
\[
 \boldsymbol\ell(Z)Z=\boldsymbol\ell(Z).
\]

The map \(Z\mapsto B(Z)\) is holomorphic, and matrix inversion is
holomorphic on the set of invertible matrices.  Hence
\[
 Z\longmapsto\boldsymbol\ell(Z)
\]
is holomorphic on \(V\).

Finally, if \(Q\in V\cap S(U)\) is real and stochastic, then
\(\boldsymbol\ell(Q)\) is a normalized left fixed vector of \(Q\).
By uniqueness of the stationary probability vector,
\[
 \boldsymbol\ell(Q)=\boldsymbol\pi(Q).
\]
\end{proof}

\begin{remark}
For non-real \(Z\), the entries \(\ell_i(Z)\) may be complex.
They are not probabilities.  The vector
\(\boldsymbol\ell(Z)\) is used only as a normalized holomorphic
functional on the finite state space.
\end{remark}


Now let
\[
 g(i,t,v)=\log\frac{\norm{A_i(t)v}}{\norm v}.
\]
Since the matrices and their inverses are uniformly bounded,
\(g\) is bounded and \(\alpha\)-Hölder in \(v\), uniformly in
\((i,t)\).

For \(n\ge0\), \(Z\in V\), and \(v\in\PP^{d-1}\), put
\begin{equation}\label{eq:direct-bn}
 b_n(Z,v)
 =
 \sum_{i=1}^N\ell_i(Z)
 \int_{\T^m}(\mathcal T_Z^ng)(i,t,v)\,dt.
\end{equation}

\begin{lemma}\label{lem:direct-holomorphic}
For every fixed \(n\) and \(v\), the function
\[
 Z\longmapsto b_n(Z,v)
\]
is holomorphic on \(V\).
\end{lemma}

\begin{proof}
The map \(Z\mapsto\boldsymbol\ell(Z)\) is holomorphic by
Lemma~\ref{lem:direct-left-vector}.  The path expansion shows that
\(\mathcal T_Z^ng\) is a homogeneous polynomial of degree \(n\) in
the supported entries of \(Z\).  Integration over the torus and the
finite sum over \(i\) preserve holomorphy.
\end{proof}


\begin{proposition}[Cauchy estimate]\label{prop:direct-cauchy}
Fix \(\alpha\in(0,\alpha_*]\), and let
\(\zeta=\zeta(\alpha)\) be given by
Proposition~\ref{prop:contraction}.  Choose
\[
 e^{-\zeta}<\gamma<1.
\]
After replacing \(V\) by a relatively compact neighborhood of \(P\)
contained in \(D_\gamma\), there is \(C_V>0\) such that
\[
 |b_{n+1}(Z,v)-b_n(Z,v)|
 \le
 C_V\bigl(\gamma^{-1}e^{-\zeta}\bigr)^n
\]
for every \(Z\in V\), \(v\in\PP^{d-1}\), and \(n\ge0\).
\end{proposition}

\begin{proof}

Starting from \eqref{eq:direct-bn} and applying the transfer operator
once gives
\[
\begin{aligned}
 b_{n+1}(Z,v)
 &=
 \sum_{i,j}\ell_i(Z)z_{ij}
 \int_{\T^m}
 (\mathcal T_Z^ng)
 \bigl(j,t+\theta_i,A_i(t)v\bigr)\,dt.
\end{aligned}
\]
Use the change of variables \(s=t+\theta_i\):
\[
\begin{aligned}
 b_{n+1}(Z,v)
 &=
 \sum_{i,j}\ell_i(Z)z_{ij}
 \int_{\T^m}
 (\mathcal T_Z^ng)
 \bigl(j,s,A_i(s-\theta_i)v\bigr)\,ds.
\end{aligned}
\]
On the other hand, the identity
\[
 \sum_i\ell_i(Z)z_{ij}=\ell_j(Z)
\]
allows us to write
\[
 b_n(Z,v)
 =
 \sum_{i,j}\ell_i(Z)z_{ij}
 \int_{\T^m}(\mathcal T_Z^ng)(j,s,v)\,ds.
\]
Subtracting,
\begin{align}
 b_{n+1}(Z,v)-b_n(Z,v)
 &=
 \sum_{i,j}\ell_i(Z)z_{ij}
 \int_{\T^m}
 \big[
 (\mathcal T_Z^ng)
   (j,s,A_i(s-\theta_i)v)\notag\\
 &\hspace{44mm}
 -(\mathcal T_Z^ng)(j,s,v)
 \big]\,ds.
\label{eq:direct-difference}
\end{align}

\smallskip
We estimate the integrand pointwise in \(s\) after expanding
\(\mathcal T_Z^n\) along paths beginning at \(j\).  If
\(u,w\in\PP^{d-1}\), then
\begin{align*}
 &
 \left|
 (\mathcal T_Z^ng)(j,s,u)
 -(\mathcal T_Z^ng)(j,s,w)
 \right|\\
 &\quad\le
 [g]_\alpha
 \sum_{\mathbf i:i_0=j}|w_Z(\mathbf i)|
 d(A_{\mathbf i}^n(s)u,A_{\mathbf i}^n(s)w)^\alpha.
\end{align*}
Here the terminal state and terminal torus point are the same in the
two terms; only the projective direction changes.  Hence the
\(\alpha\)-Hölder seminorm of \(g\) gives the displayed inequality.
For \(Z\in D_\gamma\), use
\eqref{eq:path-comparison}.  Factoring out \(d(u,w)^\alpha\) gives
\[
\begin{aligned}
 &\sum_{\mathbf i:i_0=j}|w_Z(\mathbf i)|
 d(A_{\mathbf i}^n(s)u,A_{\mathbf i}^n(s)w)^\alpha\\
 &\quad\le
 \gamma^{-n}d(u,w)^\alpha
 \sum_{\mathbf i:i_0=j}w_P(\mathbf i)
 \left(
 \frac{d(A_{\mathbf i}^n(s)u,A_{\mathbf i}^n(s)w)}
      {d(u,w)}
 \right)^\alpha.
\end{aligned}
\]
The sum with weights \(w_P\) is the integral over paths with prescribed
initial state \(j\), namely the integral against \(\mu_{P,j}\), because
\[
 \sum_{\mathbf i:i_0=j}w_P(\mathbf i)=1.
\]
By the definition of \(\widehat K_n\),
\[
\begin{aligned}
 &
 \left|
 (\mathcal T_Z^ng)(j,s,u)
 -(\mathcal T_Z^ng)(j,s,w)
 \right|\\
 &\qquad\le
 [g]_\alpha\gamma^{-n}
 \widehat K_n(\alpha,P)d(u,w)^\alpha\\
 &\qquad\le
 C(\alpha)[g]_\alpha
 \gamma^{-n}e^{-\zeta n}.
\end{aligned}
\]
This holds uniformly in \(j,s,u,w\).  We used \(d(u,w)\le1\) only in
the last line.

\smallskip
Because \(V\subset D_\gamma\), both
\[
 L_V:=\sup_{Z\in V}\sum_i|\ell_i(Z)|
\quad\text{and}\quad
 R_V:=\sup_{Z\in V}\max_i\sum_j|z_{ij}|
\]
are finite.  Applying the preceding estimate to
\[
 u=A_i(s-\theta_i)v,\qquad w=v
\]
is now legitimate for each \(s\), precisely because
\(\widehat K_n\) contains the supremum over the torus point and the
two directions.  Integrating the pointwise bound in
\eqref{eq:direct-difference} proves
\[
 |b_{n+1}(Z,v)-b_n(Z,v)|
 \le
 L_VR_VC(\alpha)[g]_\alpha
 \gamma^{-n}e^{-\zeta n}.
\]
Thus the proposition holds with
\[
 C_V=L_VR_VC(\alpha)[g]_\alpha.
\]
\end{proof}

\subsection{Uniform convergence and holomorphy}

\begin{proposition}\label{prop:direct-limit}
The sequence \((b_n(\,\cdot\,,v))_{n\ge0}\) converges uniformly on
\(V\), uniformly also in \(v\).  In particular, after fixing any
\(v_0\in\PP^{d-1}\), its limit is a holomorphic function
\[
 \Lambda:V\longrightarrow\C.
\]
\end{proposition}

\begin{proof}
Let
\[
 r=\gamma^{-1}e^{-\zeta}<1.
\]
For \(m>n\), Proposition~\ref{prop:direct-cauchy} gives
\[
\begin{aligned}
 |b_m(Z,v)-b_n(Z,v)|
 &\le\sum_{q=n}^{m-1}
       |b_{q+1}(Z,v)-b_q(Z,v)|\\
 &\le C_V\sum_{q=n}^{m-1}r^q
 \le\frac{C_Vr^n}{1-r}.
\end{aligned}
\]
The bound is independent of \(Z\in V\) and \(v\), so the sequence is
uniformly Cauchy.  Each term is holomorphic by
Lemma~\ref{lem:direct-holomorphic}.  The Weierstrass theorem implies
that the uniform limit is holomorphic.

Fix \(v_0\) and denote the resulting limit by \(\Lambda\).  The
identification below shows that on real stochastic matrices the limit
is independent of this choice.
\end{proof}

\subsection{Identification with top Lyapunov exponent for real stochastic matrices}

\begin{lemma}[Stability of projective contraction]\label{lem:real-stability}
For every \(Q\in V\cap S(U)\), projective contraction holds uniformly
in \(i,t,u,v\).  Consequently,
\[
 \frac1n\int_{\Sigma_i}
 \log\norm{A_x^n(t)v}\,d\mu_{Q,i}(x)
 \longrightarrow\lambda_+(Q)
\]
uniformly in \(i,t,v\).
\end{lemma}

\begin{proof}
Set
\[
 r=\gamma^{-1}e^{-\zeta}<1.
\]
For a \(Q\)-path of length \(n\), the defining inequality of
\(D_\gamma\) gives
\[
 q_{i_0i_1}\cdots q_{i_{n-1}i_n}
 \le\gamma^{-n}
 p_{i_0i_1}\cdots p_{i_{n-1}i_n}.
\]
Therefore the contraction coefficient computed with \(Q\) satisfies
\[
 \widehat K_n(\alpha,Q)
 \le\gamma^{-n}\widehat K_n(\alpha,P)
 \le C(\alpha)r^n.
\]

Let \(\mathcal Q_Q\) be the projective Markov kernel corresponding to
\(Q\).  If two copies start from \((i,t,u)\) and \((i,t,v)\) and are
driven by the same Markov path, then
\[
\begin{aligned}
 &\left|
 (\mathcal Q_Q^ng)(i,t,u)
 -(\mathcal Q_Q^ng)(i,t,v)
 \right|\\
 &\qquad\le
 [g]_\alpha
 \int_{\Sigma_i}
 d(A_x^n(t)u,A_x^n(t)v)^\alpha
 \,d\mu_{Q,i}(x)\\
 &\qquad\le
 [g]_\alpha C(\alpha)r^n d(u,v)^\alpha.
\end{aligned}
\]
Thus the fiber oscillation of \(\mathcal Q_Q^ng\) tends to zero
uniformly and exponentially.

Now let \(\nu_1,\nu_2\) be two stationary projective probabilities.
Their projections to \(X\times\T^m\) are both the unique stationary
base measure
\[
 \eta_Q=\sum_i\pi_i(Q)\delta_i\times\Leb.
\]
Disintegrate them as
\[
 d\nu_a(i,t,u)=d\nu_{a,i,t}(u)\,d\eta_Q(i,t),
 \qquad a=1,2.
\]
Because \(\nu_a\) is stationary,
\[
 \int g\,d\nu_a=\int\mathcal Q_Q^ng\,d\nu_a.
\]
Couple \(\nu_{1,i,t}\) and \(\nu_{2,i,t}\) arbitrarily in each fiber
and use the preceding oscillation bound.  Since \(d(u,v)\le1\),
\[
 \left|\int g\,d\nu_1-\int g\,d\nu_2\right|
 \le[g]_\alpha C(\alpha)r^n.
\]
Letting \(n\to\infty\) shows that every stationary projective
probability has the same \(g\)-integral.

The Furstenberg formula states that
\[
 \lambda_+(Q)
 =
 \max_{\nu\ {\rm stationary}}\int g\,d\nu.
\]
Therefore the common value is \(\lambda_+(Q)\).  Finally, repeat the contradiction argument (Krylov–Bogolyubov argument) from
Lemma~\ref{lem:uniform-growth}: if directional averages failed to
converge uniformly, a weak limit of violating measures
would be stationary but would have \(g\)-integral different from
\(\lambda_+(Q)\).  This proves the claimed uniform convergence.
\end{proof}

\begin{lemma}\label{lem:direct-identification}
For every real \(Q\in V\cap S(U)\),
\[
 \Lambda(Q)=\lambda_+(Q).
\]
\end{lemma}

\begin{proof}
For real \(Q\), Lemma~\ref{lem:direct-left-vector} gives
\(\boldsymbol\ell(Q)=\boldsymbol\pi(Q)\).  Expanding the paths in
\eqref{eq:direct-bn} therefore yields
\[
\begin{aligned}
 b_n(Q,v)
 &=
 \sum_{i_0,\ldots,i_n}
 \pi_{i_0}(Q)
 q_{i_0i_1}\cdots q_{i_{n-1}i_n}\\
 &\qquad\qquad\times
 \int_{\T^m}
 \log
 \frac{
 \norm{
 A_{i_n}(t_n)A_{\mathbf i}^n(t)v}}
 {\norm{A_{\mathbf i}^n(t)v}}
 \,dt.
\end{aligned}
\]
therefore 
\[
 b_n(Q,v)
 =
 \int_{\Sigma\times\T^m}
 \log
 \frac{\norm{A_x^{n+1}(t)v}}
      {\norm{A_x^n(t)v}}
 \,d\mu_Q(x)\,dt.
\]

The increments telescope:
\[
\begin{aligned}
 \frac1n\sum_{q=0}^{n-1}b_q(Q,v)
 &=
 \frac1n
 \int_{\Sigma\times\T^m}
 \log\frac{\norm{A_x^n(t)v}}{\norm v}
 \,d\mu_Q(x)\,dt\\
 &\longrightarrow\lambda_+(Q)
\end{aligned}
\]
by Lemma~\ref{lem:real-stability}.
Since:
\[
 \sum_{q=0}^{n-1}
 \log\frac{\norm{A_x^{q+1}(t)v}}
          {\norm{A_x^q(t)v}}
 =
 \log\frac{\norm{A_x^n(t)v}}{\norm v}.
\]

Proposition~\ref{prop:direct-limit} gives
\[
 b_n(Q,v)\longrightarrow\Lambda(Q).
\]
A convergent sequence and its Cesàro averages have the same limit.
Consequently \(\Lambda(Q)=\lambda_+(Q)\).
\end{proof}

\begin{proof}[Proof of Theorem~\ref{thm:irreducible-core}]
By Proposition~\ref{prop:direct-limit}, \(\Lambda\) is holomorphic on
a complex neighborhood \(V\subset H(U)\) of \(P\).
Lemma~\ref{lem:direct-identification} shows that its restriction to
the real stochastic matrices equals the top Lyapunov exponent.  Hence
\(Q\mapsto\lambda_+(Q)\) is real-analytic near \(P\).

\end{proof}

\section{Invariant subbundles}

\subsection{Restricted and quotient Markov bundle cocycles}

We now remove irreducibility by the same dimension reduction argument
used for the Bernoulli model in~\cite[Section~3 and the proof of
Theorem~1]{BST}.  The only structural change is that an invariant
family may retain the current Markov state.

Let
\[
 \mathcal V_i:\T^m\longrightarrow\Gr(k,d),
 \qquad i\in X,
\]
be measurable and suppose that
\begin{equation}\label{eq:markov-invariant-family}
 A_i(t)\mathcal V_i(t)
 =\mathcal V_j(t+\theta_i)
 \quad\text{for every }(i,j)\in\cE
\end{equation}
for almost every \(t\).  Since every matrix \(Q\in S(U)\) has exactly
the same admissible edges, \(\boldsymbol{\mathcal V}\) is invariant
for every \(Q\in S(U)\), not merely for the reference matrix \(P\).
This is the fixed-support analogue of the assumption \(p_i>0\) for
every symbol in the Bernoulli paper.

Choose measurable orthonormal trivializations
\[
 J_i(t):\mathcal V_i(t)\longrightarrow\R^k.
\]
The restricted bundle maps along admissible edges are
\[
 A^{\mathcal V}_{ij}(t)
 =
 J_j(t+\theta_i)\circ
 A_i(t)|_{\mathcal V_i(t)}
 \circ J_i(t)^{-1}.
\]
Because the trivializations are isometries and
\eqref{eq:markov-invariant-family} is an equality of subspaces,
\begin{equation}\label{eq:restricted-essential-bounds}
 \|A^{\mathcal V}_{ij}(t)\|\leq\|A_i(t)\|,
 \qquad
 \|(A^{\mathcal V}_{ij}(t))^{-1}\|
 \leq\|A_i(t)^{-1}\|
\end{equation}
for almost every \(t\).  In particular, the restriction belongs to
\(\mathcal B_k(U)\).
Similarly, using the orthogonal quotient
\(\mathcal V_i(t)^\perp\simeq\R^d/\mathcal V_i(t)\), we obtain
quotient maps \(A^{\mathrm{quot}}_{ij}(t)\) of dimension \(d-k\).
Equipping each quotient with its quotient norm gives
\begin{equation}\label{eq:quotient-essential-bounds}
 \|A^{\mathrm{quot}}_{ij}(t)\|\leq\|A_i(t)\|,
 \qquad
 \|(A^{\mathrm{quot}}_{ij}(t))^{-1}\|
 \leq\|A_i(t)^{-1}\|.
\end{equation}
Thus the quotient belongs to \(\mathcal B_{d-k}(U)\).  These maps may
be merely measurable in \(t\), even though every \(A_i\) is
continuous.  Different measurable
trivializations give measurably conjugate cocycles and hence the same
Lyapunov exponents.

This is the regularity point that must be handled in the reduction.
Applying a theorem stated only on
\(C^0(\mathbb T^m,\operatorname{GL}_r(\mathbb R))\) therefore requires
the measurable bounded extension recorded below.
It is natural to regard these as cocycles on the measurable vector
bundle over \(X\times\T^m\).  If one uses coordinates, the associated
transfer operator is
\[
 (\mathcal T_Z^{\mathcal V}\Phi)(i,t, v)
 =
 \sum_{j:(i,j)\in\cE}z_{ij}
 \Phi\bigl(j,t+\theta_i,
       A^{\mathcal V}_{ij}(t)v\bigr).
\]
Thus the finite path expansions remain polynomial in the entries of
\(Z\), exactly as in the irreducible proof.

\begin{proposition}[Furstenberg-Kifer]\label{prop:block-formula}
For every \(Q\in S(U)\),
\[
 \lambda_+(\mathbf A,Q)
 =\max\left\{
 \lambda_+(\mathbf A^{\boldsymbol{\mathcal V}},Q),
 \lambda_+(\mathbf A^{\mathrm{quot}},Q)
 \right\}.
\]
More generally, the Lyapunov spectrum of \(\mathbf A\), including
multiplicities, is the union of the spectra of the restriction and
the quotient.
\end{proposition}

\begin{proof}
This is The Furstenberg--Kifer block triangular lemma
\cite[Lemma~III.3.3]{Kifer}.  Applying the same reasoning to exterior
powers, or equivalently using the exact sequence of Oseledets
filtrations associated with the invariant subbundle, shows that the
full spectrum is the union of the two diagonal spectra with
multiplicity.
\end{proof}

\subsection{The irreducible case for measurable bundles}

The dimension reduction produces measurable bundle maps, even though
the original \(A_i\)'s are continuous.  We therefore record explicitly
the bundle form of the irreducible case.

\begin{proposition}
\label{prop:bundle-core}
Let \(\mathbf B\in\mathcal B_r(U)\), let \(P\in S(U)\) be primitive,
and suppose Assumption~\ref{ass:cycle} holds.  If the measurable bundle
cocycle is irreducible and
\[
 \lambda_1(\mathbf B,P)>\lambda_2(\mathbf B,P),
\]
then there are a relatively open neighborhood
\(\Omega\subset H(U)\) of \(P\) and a holomorphic function
\(\Lambda_{\mathbf B}:\Omega\to\mathbb C\) such that
\[
 \Lambda_{\mathbf B}(Q)=\lambda_+(\mathbf B,Q)
\]
for every \(Q\in\Omega\cap S(U)\).
\end{proposition}

\begin{proof}
This is not a formal corollary of
Theorem~\ref{thm:irreducible-core}.  For a measurable cocycle the
projective transition kernel need not be a continuous Markov operator, its logarithmic
observable need not be continuous, and weak convergence of
measures does not control the integral of that observable.
We therefore replace the
uniform in \(t\)  argument by the torus integrated
measurable argument.

\smallskip
\noindent\emph{1. The measurable Furstenberg--Kifer input.}
The non-random filtration theorem is formulated for measurable bundle
maps; see Kifer~\cite[Chapter~III]{Kifer}.  Since the one-step maps and
their inverses are essentially bounded, all logarithmic integrability
hypotheses are automatic.  Irreducibility eliminates every proper
term of the non-random filtration.  Consequently, for every measurable
projective section \(\xi=(\xi_i(t))\),
\[
 \frac1n\sum_i\pi_i(P)
 \int_{\T^m}
 \int_{\Sigma_i}
 \log\norm{B_x^n(t)\xi_i(t)}
 \,d\mu_{P,i}(x)\,dt
 \longrightarrow\lambda_1(B,P),
\]
uniformly over unit measurable sections.  This is the measurable
bundle analogue of Lemma~\ref{lem:uniform-growth}; in the Bernoulli
case it is the input used in \cite[Proposition~7]{BST}.

For later use, introduce the integrated fiber-contraction coefficient
\[
\begin{aligned}
 \mathfrak K_n(\alpha,P)
 =\max_i\sup_{\xi,\eta}
 \int_{\T^m}\int_{\Sigma_i}
 \left(
 \frac{
 d(B_x^n(t)\xi(t),B_x^n(t)\eta(t))}
 {d(\xi(t),\eta(t))}
 \right)^\alpha
 \,d\mu_{P,i}(x)\,dt,
\end{aligned}
\]
where the supremum is over measurable projective sections that are
distinct almost everywhere.  The measurable bundle version of the
Furstenberg--Kifer contraction theorem gives
\begin{equation}\label{eq:bundle-contraction}
 \mathfrak K_n(\alpha,P)
 \le C_{\mathrm b}e^{-\zeta_{\mathrm b}n}
\end{equation}
for some sufficiently small \(\alpha>0\).  To see the mechanism,
repeat the three stages of Proposition~\ref{prop:contraction}: the
non-random filtration gives uniform averaged growth of every
measurable direction, the exterior-square estimate gives a negative
mean for the logarithm of projective distortion, and the inequality
\[
 e^s\le1+s+\frac12s^2e^{|s|}
\]
turns one sufficiently long block into a contracting fractional
moment.  All conditional path integrals are estimated uniformly over
the finite current-state index.  Formula
\eqref{eq:bundle-contraction} is intrinsic to the bundle and is
unchanged by measurable orthonormal trivializations.

\smallskip
\noindent\emph{2. The logarithmic observable.}
In an orthonormal trivialization set
\[
 g_{ij}(t,\bar v)
 =
 \log\frac{\norm{B_{ij}(t)v}}{\norm v}.
\]
If
\[
 M_{\mathrm b}
 =\operatorname*{ess\,sup}_{i,j,t}\norm{B_{ij}(t)},
 \qquad
 m_{\mathrm b}
 =\operatorname*{ess\,inf}_{i,j,t}
 \norm{B_{ij}(t)^{-1}}^{-1},
\]
then
\[
 |g_{ij}|\le
 \max\{|\log M_{\mathrm b}|,|\log m_{\mathrm b}|\},
\]
and the fiber Hölder seminorm is bounded uniformly in \(i,j,t\).
No continuity in \(t\) is used here.

\smallskip
\noindent\emph{3. Finite complex iterates.}
For every \(n\), the bundle transfer iterate is the finite path sum
\[
 (\mathcal T_Z^n\Phi)(i,t,v)
 =
 \sum_{i=i_0,\ldots,i_n}
 z_{i_0i_1}\cdots z_{i_{n-1}i_n}
 \Phi(i_n,t_n,B_{\mathbf i}^n(t)v).
\]
Thus, after integration in \(t\), it is a polynomial in the supported
entries of \(Z\).  Bounded measurability of the integrand is enough:
holomorphy concerns the finite-dimensional parameter \(Z\), not the
base variable \(t\).

\smallskip
\noindent\emph{4. The Cauchy estimate.}
The analogue of \eqref{eq:direct-difference} contains two measurable
projective sections of \(t\).  Apply
\eqref{eq:path-comparison} to the complex coefficients and
\eqref{eq:bundle-contraction} to those two sections.  One obtains
\[
 |b_{n+1}^{\mathrm b}(Z)-b_n^{\mathrm b}(Z)|
 \le
 C_V^{\mathrm b}
 \bigl(\gamma^{-1}e^{-\zeta_{\mathrm b}}\bigr)^n,
\]
provided \(e^{-\zeta_{\mathrm b}}<\gamma<1\).
The right-hand side is summable.  Hence the bundle increments converge
uniformly on a complex neighborhood of \(P\), and their limit is
holomorphic by Weierstrass.  On real stochastic matrices the same
telescoping calculation as in
Lemma~\ref{lem:direct-identification} identifies the limit with the
top bundle exponent.

This proves the proposition.  The argument uses measurable
orthonormal coordinates only to write formulas; neither the result nor
the constants depend on the chosen coordinates.
\end{proof}

\begin{proposition}[Continuity on a fixed support]
\label{prop:block-continuity}
For every essentially bounded measurable Markov bundle cocycle with
essentially bounded inverse, its top Lyapunov exponent is continuous
as a function of \(Q\in S(U)\).
\end{proposition}

\begin{proof}
The proof is by induction on the fiber dimension, as in
\cite[Proposition~4]{BST}.

\smallskip
\noindent\emph{Dimension one.}
Write a one-dimensional edge map as multiplication by
\(a_{ij}(t)\ne0\).  Then
\[
 \lambda_+(Q)
 =
 \sum_{i,j}\pi_i(Q)q_{ij}
 \int_{\T^m}\log|a_{ij}(t)|\,dt.
\]
The integrals are finite constants, while
\(\pi_i(Q)\) and \(q_{ij}\) depend continuously on \(Q\).  Thus
\(\lambda_+\) is continuous.

\smallskip
\noindent\emph{The irreducible case in dimension \(r>1\).}
Fix a measurable unit projective section \(\xi\) and define
\[
 L_n(Q,\xi)
 =
 \frac1n\sum_i\pi_i(Q)
 \int_{\T^m}\int_{\Sigma_i}
 \log\norm{B_x^n(t)\xi_i(t)}
 \,d\mu_{Q,i}(x)\,dt.
\]
For fixed \(n\), expansion over the finitely many admissible paths
shows that \(Q\mapsto L_n(Q,\xi)\) is continuous: the path weights and
\(\boldsymbol\pi(Q)\) are continuous, while all integrals are fixed
bounded numbers.  The Furstenberg--Kifer theorem gives
\[
 L_n(Q,\xi)\longrightarrow\lambda_+(Q)
\]
uniformly in \(\xi\), and locally uniformly for \(Q\) in the relative
interior of the fixed support face.  A locally uniform limit of
continuous functions is continuous.  This proves continuity at every
irreducible \(Q\).

\smallskip
\noindent\emph{The reducible case.}
Let \(\boldsymbol{\mathcal V}\) be an invariant family.  It depends
only on the bundle maps and on the support \(\cE\), so it remains
invariant for every \(Q\in S(U)\).  Proposition~\ref{prop:block-formula}
gives
\[
 \lambda_+(Q)
 =
 \max\{\lambda_{\mathcal V}(Q),
        \lambda_{\mathrm{quot}}(Q)\}.
\]
Both block dimensions are strictly smaller than \(r\), hence both
functions on the right are continuous by the induction hypothesis.
Their maximum is continuous.  This closes the induction.
\end{proof}

\subsection{Proof of the main theorem}

\begin{proof}[Proof of Theorem~\ref{thm:main}]
We induct on the fiber dimension \(d\).

\smallskip
\noindent\emph{Base of the induction.}
In dimension one, if the fiber map along \(i\to j\) is multiplication
by \(a_{ij}(t)\), then
\[
 \lambda_+(Q)
 =
 \sum_{i,j}\pi_i(Q)q_{ij}
 \int_{\T^m}\log|a_{ij}(t)|\,dt.
\]
The stationary vector has a holomorphic extension near \(P\), and the
remaining dependence on \(Q\) is polynomial.  Hence the exponent has
a holomorphic extension.  If \(d>1\) and the cocycle is irreducible,
Theorem~\ref{thm:irreducible-core} proves the assertion.

\smallskip
\noindent\emph{First reduction.}
Suppose that the cocycle is reducible and choose a proper invariant
family \(\boldsymbol{\mathcal V}\).  Write
\[
 \lambda_{\mathcal V}(Q)
 =\lambda_+(\mathbf A^{\boldsymbol{\mathcal V}},Q),
 \qquad
 \lambda_{\mathrm{quot}}(Q)
 =\lambda_+(\mathbf A^{\mathrm{quot}},Q).
\]
By Proposition~\ref{prop:block-formula},
\[
 \lambda_+(\mathbf A,Q)
 =\max\{\lambda_{\mathcal V}(Q),
        \lambda_{\mathrm{quot}}(Q)\}.
\]
At \(P\), these two numbers cannot be equal
because \(\lambda_1(P)>\lambda_2(P)\).

assume
\[
 \lambda_{\mathcal V}(P)>
 \lambda_{\mathrm{quot}}(P)
\]
and define
\[
 \delta=
 \frac13\bigl(
 \lambda_{\mathcal V}(P)
 -\lambda_{\mathrm{quot}}(P)\bigr)>0.
\]
By Proposition~\ref{prop:block-continuity}, there is a real
neighborhood \(W\subset S(U)\) of \(P\) such that
\[
 |\lambda_{\mathcal V}(Q)-\lambda_{\mathcal V}(P)|<\delta,
 \qquad
 |\lambda_{\mathrm{quot}}(Q)-\lambda_{\mathrm{quot}}(P)|<\delta
\]
for every \(Q\in W\).  Consequently,
\[
\begin{aligned}
 \lambda_{\mathcal V}(Q)-\lambda_{\mathrm{quot}}(Q)
 &>
 \lambda_{\mathcal V}(P)-\lambda_{\mathrm{quot}}(P)-2\delta\\
 &=\delta>0.
\end{aligned}
\]
Therefore the same block remains dominant:
\[
 \lambda_+(\mathbf A,Q)=\lambda_{\mathcal V}(Q),
 \qquad Q\in W.
\]

\smallskip
\noindent\emph{Iteration.}
If the restricted bundle cocycle is irreducible,
Proposition~\ref{prop:bundle-core} gives a holomorphic extension of
\(\lambda_{\mathcal V}\), and hence of
\(\lambda_+(\mathbf A,\cdot)\), near \(P\).  If it is reducible,
repeat the construction inside this dominant block.

The top exponent of the dominant block is still \(\lambda_1(P)\) and
is still simple. 
Each step selects one strictly dominant block and its
dimension decreases:
\[
 d>d_1>d_2>\cdots.
\]
After at most \(d-1\) steps, the process terminates at an irreducible
measurable bundle cocycle \(\mathbf B\).

Only finitely many continuity neighborhoods were chosen.  Their
intersection contains a real neighborhood \(W_*\) of \(P\) on which
\[
 \lambda_+(\mathbf A,Q)=\lambda_+(\mathbf B,Q).
\]
Proposition~\ref{prop:bundle-core} supplies a complex neighborhood
\(\Omega\subset H(U)\) and a holomorphic function
\(\Lambda_{\mathbf B}\) satisfying
\[
 \Lambda_{\mathbf B}(Q)=\lambda_+(\mathbf B,Q)
\]
for real \(Q\in\Omega\cap S(U)\).  Shrink \(\Omega\), if necessary,
so that \(\Omega\cap S(U)\subset W_*\).  Then
\[
 \Lambda_{\mathbf B}(Q)=\lambda_+(\mathbf A,Q),
 \qquad Q\in\Omega\cap S(U),
\]
which proves the theorem.
\end{proof}

\section*{Acknowledgments}
The author thanks Jennifer Loria and Marcelo Viana for helpful discussions.
This work was partially supported by Fundação Getulio Vargas, and CIMPA (Centro de Investigación en Matemática Pura y Aplicada ) University of Costa Rica (UCR).

\end{document}